\documentclass[preprint,12pt]{elsarticle}

\usepackage{amssymb}
\usepackage{amsmath}

\usepackage{hyperref}
\usepackage{algorithm,algpseudocode,amsthm,listings,booktabs,subcaption,multirow, makecell}
\usepackage{tikz, url}
\usetikzlibrary{shapes.geometric, arrows.meta, positioning, calc, shadows, fit}
 
\newtheorem{proposition}{Proposition}
\newtheorem{corollary}{Corollary}
 
\begin{document}

\begin{frontmatter}



\title{\texttt{deepbullwhip}: An Open-Source Simulation and Benchmarking for Multi-Echelon Bullwhip Analyses} 

\author{Mansur M. Arief} 

\affiliation{organization={Industrial and Systems Engineering\\Interdisciplinary Research Center for Smart Mobility and Logistics (IRC-SML)\\King Fahd University of Petroleum and Minerals (KFUPM)},
            country={Saudi Arabia}
            }

\begin{abstract}
The bullwhip effect remains operationally persistent despite decades of analytical research. Two computational deficiencies hinder progress: the absence of modular open-source simulation tools for multi-echelon inventory dynamics with asymmetric costs, and the lack of a standardized benchmarking protocol for comparing mitigation strategies across shared metrics and datasets. This paper introduces \texttt{deepbullwhip}, an open-source Python package that integrates a simulation engine for serial supply chains (with pluggable demand generators, ordering policies, and cost functions via abstract base classes, and a vectorized Monte Carlo engine achieving 50 to 90 times speedup) with a registry-based benchmarking framework shipping a curated catalog of ordering policies, forecasting methods, six bullwhip metrics, and demand datasets including WSTS semiconductor billings. Five sets of experiments on a four-echelon semiconductor chain demonstrate cumulative amplification of $427\times$ (Monte Carlo mean across $1{,}000$ paths), a stochastic filtering phenomenon at upstream tiers (CV $= 0.01$), super-exponential lead time sensitivity, and scalability to 20.8 million simulation cells in under 7 seconds. Benchmark experiments reveal a $155\times$ disparity between synthetic AR(1) and real WSTS bullwhip severity under the Order-Up-To policy, and quantify the BWR--NSAmp tradeoff across ordering policies, demonstrating that no single metric captures policy quality.
\end{abstract}


\end{frontmatter}

\section{Introduction}\label{sec:introduction}

The bullwhip effect, the tendency for order variability to increase as one moves upstream in a supply chain, has been a central topic in operations management since \citet{Lee1997} identified its four root causes. \citet{Chen2000} derived closed-form lower bounds on the bullwhip ratio under the Order-Up-To (OUT) policy, establishing that amplification grows quadratically in lead time. Subsequent analytical work has extended these results to proportional ordering policies \citep{Disney2003}, transfer function characterizations \citep{Dejonckheere2003}, correlated stochastic lead times \citep{Disney2025}, and forecast uncertainty metrics \citep{Kourentzes2025}. Despite this substantial analytical progress, the bullwhip effect remains stubbornly persistent in practice. \citet{Brauch2024} systematically reviewed over 100 studies to conclude that demand signal processing and lead time variability continue to dominate as root causes. The 2020--2022 global semiconductor shortage provided a vivid illustration at industry scale: a moderate surge in demand for remote-work electronics triggered successive rounds of panic ordering and capacity double-booking, followed by a supply glut that forced double-digit price corrections in the memory segment by 2023 \citep{Bain2024}.

A key reason for this disconnect between theory and practice is that real supply chains violate the assumptions underlying most analytical results: stationary demand, single-echelon isolation, known distributional parameters, and unlimited upstream capacity \citep{Wang2016}. When demand exhibits structural breaks, regime switching, heavy tails, or when machine learning forecasters with non-standard error distributions replace simple moving averages \citep{Oroojlooyjadid2020,Salinas2020deepar}, simulation becomes the primary tool for evaluating multi-echelon bullwhip dynamics and their cost consequences \citep{Dominguez2014,Sagawa2024}. The need for scalable, reproducible simulation is especially acute in the machine learning era, where new forecasters must be evaluated not on point accuracy alone but on their downstream effect on multi-echelon cost and variance amplification \citep{Ban2019}. At the same time, the growing interest in digital twin approaches for supply chain resilience \citep{Corsini2024} underscores the demand for modular, open simulation engines that can serve as the computational backbone of what-if analysis and scenario planning.

However, the current landscape of computational tools falls short of these needs. Existing simulation environments are either proprietary and license-restricted (AnyLogic, Arena), pedagogical in scope with fixed structure (the Beer Distribution Game of \citet{Sterman1989}), or general-purpose supply chain libraries that address related but different problems. The Python library \texttt{supplychainpy} \citep{supplychainpy}, for instance, provides demand classification, ABC/XYZ analysis, and economic order quantity calculations, but does not model multi-echelon bullwhip propagation or newsvendor costs. The open-source platform \texttt{frePPLe} \citep{frepple} targets production scheduling and capacity planning rather than inventory policy evaluation. As a result, researchers studying the bullwhip effect typically build
ad-hoc simulation scripts that are tightly coupled to specific
assumptions, difficult to replicate, and hard to extend without
rewriting from scratch \citep{Wang2016,Brauch2024}. A related and
arguably more consequential shortcoming is the absence of a standardized
benchmarking protocol. Studies select metrics inconsistently
(bullwhip ratio alone, or Net Stock Amplification alone, or total
cost), test against different demand processes, and compare against
different baseline policies, making cross-study comparison unreliable
\citep{ChenLee2012,Kourentzes2025}. The forecasting literature addressed such an analogous problem through the M-competitions \citep{Makridakis2022}, which established shared datasets and evaluation protocols that enabled cumulative learning across hundreds of methods. No equivalent infrastructure exists for the bullwhip effect studies.

This paper addresses both deficiencies through \texttt{deepbullwhip}, an open-source Python package comprising twofold contributions. \textbf{The first} is a \emph{simulation engine} for serial supply chains in which demand generators, ordering policies, cost functions, and forecasting methods are all substitutable through abstract base classes, together with a vectorized Monte Carlo engine that exploits NumPy broadcasting to process over 20 million simulation cells in under 7 seconds on a single CPU core. \textbf{The second} is a \emph{benchmarking framework} consisting of a decorator-based model registry, a curated catalog of ordering policies, forecasting methods, bullwhip metrics, and demand datasets (including WSTS semiconductor billings \citep{wsts2025} and M5 Walmart retail data \citep{Makridakis2022}), making it easy to produce standardized comparison tables across any combination of registered components. We validate the simulator against the analytical bounds of \citet{Chen2000} and present five sets of experiments on a four-echelon semiconductor supply chain. The simulation experiments demonstrate cumulative bullwhip amplification of $427\times$, a variance ratio filtering phenomenon at upstream tiers (Proposition~\ref{prop:filtering}), and super-exponential sensitivity of cumulative amplification to foundry lead time. The benchmark experiments compare ordering policies and forecasting methods across multiple metrics simultaneously, and reveal a 155-fold disparity between synthetic and real-data bullwhip severity that illustrates the practical importance of empirical benchmarking. The package, with full documentation is released open-source at \url{https://github.com/ai-vnv/deepbullwhip}.

The remainder of this paper is organized as follows. Section~\ref{sec:background} reviews the analytical foundations and existing tools. Section~\ref{sec:engine} describes the simulation engine. Section~\ref{sec:benchmark} presents the benchmarking framework. Section~\ref{sec:experiments} reports the computational experiments. Section~\ref{sec:discussion} discusses findings and implications, and Section~\ref{sec:conclusion} concludes.

\section{Literature review}\label{sec:background}

The bullwhip effect has been studied through analytical, empirical, experimental, and simulation-based methods since \citet{Forrester1961} first observed demand amplification in system dynamics models. \citet{Wang2016} provided a comprehensive review of progress, trends, and open directions. This section organizes the relevant literature around three themes that frame the contributions of the present paper: the analytical characterization of bullwhip through metrics and ordering policies (Section~\ref{sec:bg_analytical}), the role of forecasting methods and their interaction with cost performance (Section~\ref{sec:bg_forecasting}), and the state of computational tools for bullwhip simulation (Section~\ref{sec:bg_tools}).

\subsection{Bullwhip metrics, ordering policies, and demand models}\label{sec:bg_analytical}

The standard bullwhip metric is the variance ratio (BWR) \citep{Chen2000}
\begin{equation}\label{eq:bwr}
    \text{BWR}_k = \frac{\text{Var}(O_k)}{\text{Var}(D_k)},
\end{equation}
where $O_k$ and $D_k$ denote orders placed and demand received at echelon $k$, respectively, with a variance ratio of 1 for a perfectly balanced chain. Furthermore, \citet{Chen2000} derived the foundational lower bound
\begin{equation}\label{eq:bwr_lower_bound}
    \text{BWR}_k \geq 1 + \frac{2(L_k+1)}{p} + \frac{2(L_k+1)^2}{p^2},
\end{equation}
for the Order-Up-To (OUT) policy with moving-average forecasting of window $p$, where $L_k + 1$ is the review lead time (lead time plus one review period). The bound reveals quadratic dependence on lead time, which is the primary reason why semiconductor chains with foundry lead times of 12 weeks or more \citep{monch2013} exhibit far stronger bullwhip than consumer goods chains. \citet{Zhang2004} extended the analysis to minimum mean squared error (MMSE) forecasting under ARMA demand, and \citet{Gilbert2005} studied ARIMA demand processes. More recently, \citet{Gaalman2022} used eigenvalue analysis to characterize when BWR is an increasing function of lead time under general ARMA demand, showing that this property depends on the sign of the demand impulse response rather than on the demand model parameters alone.

\citet{Disney2003} introduced the Net Stock Amplification ratio 
\begin{equation}\label{eq:nsamp}
    \text{NSAmp} = \frac{\text{Var}(I_k)}{\text{Var}(D_k)},
\end{equation}
where $I_k$ is net inventory and demonstrated a fundamental tradeoff: policies that minimize BWR do not necessarily minimize NSAmp, and vice versa. This tradeoff motivates the Proportional Order-Up-To (POUT) policy 
\begin{equation}\label{eq:pout}
    O_k(t) = \alpha \cdot [S_k(t) - \text{IP}_k(t)]^+,
\end{equation}
where $\alpha \in (0,1]$ is a smoothing parameter that interpolates between full demand chasing ($\alpha = 1$, standard OUT) and complete order smoothing ($\alpha \to 0$), and $[\cdot]^+ = \max\{0, \cdot\}$ denotes the positive part. Note that $\alpha$ scales the non-negative gap, so the truncation at zero applies before the smoothing; this follows the formulation in \citet{Disney2003}. \citet{Dejonckheere2003} analyzed the POUT and related policies using transfer function methods and showed that order smoothing reduces peak amplification at the cost of slower inventory recovery. \citet{Cannella2021} extended the POUT analysis to closed-loop supply chains with returns, showing that tuning the inventory controllers can yield large cost savings even when return uncertainty is present. \citet{LiDisney2023} established the equivalence between the POUT and damped-trend OUT policies through eigenvalue analysis, unifying two previously separate streams of research.

Despite its prominence, several authors have questioned whether the standard BWR is the most appropriate metric. \citet{ChenLee2012} showed that the BWR can be completely non-informative about underlying supply chain cost performance when it is not linked to operational details such as decision intervals and lead times. \citet{Kourentzes2025} proposed the Ratio of Forecast Uncertainty (RFU), which compares forecast error variance across echelons and correlates more strongly with inventory costs than the standard BWR. \citet{Brauch2024} reviewed the causes of the bullwhip effect across more than 100 studies and categorized them systematically, concluding that demand signal processing and lead time variability remain the dominant drivers and that measurement inconsistency across studies is itself a barrier to cumulative progress.

\subsection{Forecasting methods and their cost consequences}\label{sec:bg_forecasting}

The choice of forecasting method interacts with the bullwhip effect in non-obvious ways. \citet{Wright2008} compared Holt's method and Brown's double exponential smoothing against standard moving averages and found that the more sophisticated methods substantially reduce BWR, but the cost implications depend on the chain configuration. \citet{Ban2019} demonstrated a broader disconnect between point forecast accuracy and newsvendor cost performance, showing that machine learning methods can improve accuracy by 20\% yet yield no improvement or even worsening of inventory cost when the cost structure is highly asymmetric ($b \gg h$). \citet{Oroojlooyjadid2020} addressed this directly by training neural networks to minimize newsvendor cost rather than forecast error, bypassing the forecast-then-optimize pipeline. \citet{Salinas2020deepar} introduced DeepAR, a probabilistic forecasting method based on autoregressive recurrent networks that produces full predictive distributions rather than point forecasts, enabling direct estimation of safety stock requirements under the newsvendor cost criterion.

These developments highlight a key limitation of the existing bullwhip literature: most analytical results assume that forecasting is exogenous and evaluate the bullwhip ratio in isolation from inventory cost. In a multi-echelon setting, however, the choice of forecaster at the downstream tier affects the demand signal seen by upstream tiers, creating cascading interactions that can only be characterized through simulation. The asymmetric newsvendor cost function penalizes stockouts disproportionately when $b_k \gg h_k$, and the bullwhip effect exacerbates stockout risk by amplifying order variability upstream. Evaluating forecasters on BWR alone, or on point accuracy alone, therefore misses the operational consequence of forecast choice, motivating evaluation across multiple metrics simultaneously.

\subsection{Existing computational tools}\label{sec:bg_tools}

Table~\ref{tab:tools} summarizes the landscape of computational tools available for bullwhip research. Proprietary platforms such as AnyLogic \citep{Borshchev2013} 
and Arena \citep{Kelton2015} provide rich simulation 
environments with multi-echelon support, but their license 
costs, closed-source codebases, and graphical modeling 
paradigms limit reproducibility and programmatic extension. The Beer Distribution Game \citep{Sterman1989} remains the most widely used educational tool, but its fixed four-echelon structure, deterministic demand step, and absence of cost optimization make it unsuitable for research benchmarking. General-purpose open-source supply chain libraries such as \texttt{supplychainpy} \citep{supplychainpy} and \texttt{frePPLe} \citep{frepple} do not model multi-echelon bullwhip propagation or evaluate newsvendor costs. The predominant approach in the research literature is ad-hoc simulation: custom scripts in MATLAB, Python, or R that implement a specific demand model, a specific ordering policy, and a specific set of metrics, with no standardized interfaces for substitution or comparison.

\begin{table*}[htbp]
\centering
\caption{Comparison of computational tools for bullwhip effect research.}\label{tab:tools}
\resizebox{\linewidth}{!}{
\begin{tabular}{lcccccc}
\toprule
Tool & \makecell{Multi-\\echelon} & \makecell{Scalable\\Monte Carlo} & \makecell{Pluggable\\policies} & \makecell{Newsvendor\\cost} & \makecell{Open\\source} & \makecell{Benchmark\\framework} \\
\midrule
AnyLogic & \checkmark & \checkmark & & & & \\
Arena & \checkmark & \checkmark & & & & \\
MIT Beer Game & \checkmark & & & & & \\
\texttt{supplychainpy} & & & & & \checkmark & \\
\texttt{frePPLe} & & & \checkmark & & \checkmark & \\
Ad-hoc scripts & varies & & varies & varies & varies & \\
\midrule
\texttt{deepbullwhip} & \checkmark & \checkmark & \checkmark & \checkmark & \checkmark & \checkmark \\
\bottomrule
\end{tabular}
}
\end{table*}

Two clear gaps emerge from this review. First, no existing open-source tool combines multi-echelon simulation, pluggable policy and forecaster components, asymmetric newsvendor cost evaluation, and scalable Monte Carlo computation. Second, the bullwhip literature lacks a standardized benchmarking protocol that evaluates policies and forecasters across multiple metrics (BWR, NSAmp, fill rate, total cost) on shared demand datasets and chain configurations. \texttt{deepbullwhip} addresses both deficiencies through a modular simulation engine and a registry-based benchmarking framework, as described in the following two sections.

\section{Simulation engine}\label{sec:engine}

This section describes the simulation engine that constitutes the first contribution of the paper. We present the mathematical models underlying the simulation (Section~\ref{sec:models}), the modular software architecture (Section~\ref{sec:architecture}), and the vectorized Monte Carlo engine (Section~\ref{sec:vectorized}).

\subsection{Mathematical models}\label{sec:models}

Consider a serial supply chain with $K$ echelons indexed $k = 1, \ldots, K$, where $k = 1$ is the echelon closest to the end customer. Time is discrete and indexed by $t = 1, \ldots, T$. The state of echelon $k$ at time $t$ comprises on-hand inventory $I_k(t)$, a pipeline of orders in transit $P_k(\tau)$ for $\tau = 1, \ldots, L_k$ where $L_k$ is the deterministic lead time, and a cost function parameterized by holding cost $h_k$ and backorder cost $b_k$. The simulation proceeds through three stages at each period: demand arrives, ordering decisions are made, and costs are incurred.

\paragraph{Demand.} End-customer demand at echelon $k=1$ follows a configurable stochastic process. The default is an AR(1) process with additive seasonality and a structural shock:
\begin{equation}\label{eq:demand}
D(t) = \mu + \phi\bigl(D(t-1) - \mu\bigr) + A\sin\!\left(\frac{2\pi t}{52}\right) + \delta(t) + \varepsilon(t),
\end{equation}
where $\mu$ is the long-run mean, $\phi \in [0,1)$ is the autoregressive coefficient, $A$ is the seasonal amplitude, $\varepsilon(t) \sim \mathcal{N}(0, \sigma^2)$ is white noise, and $\delta(t)$ is a one-time structural shock defined as
\begin{equation}\label{eq:shock}
\delta(t) = \begin{cases} M \cdot \mu & \text{if } t \geq t^*, \\ 0 & \text{otherwise}, \end{cases}
\end{equation}
with magnitude $M$ at period $t^*$. Upstream echelons do not observe end-customer demand directly; instead, each echelon $k > 1$ faces the orders placed by its downstream neighbor as its incoming demand:
\begin{equation}\label{eq:cascade}
D_k(t) = O_{k-1}(t) \quad \text{for } k > 1.
\end{equation}
This demand cascade is the mechanism through which order variability propagates upstream.

\paragraph{Ordering.} Given the incoming demand, each echelon must decide how much to order from its own supplier. Under the default OUT policy, echelon $k$ sets an order-up-to level that covers expected demand over the lead time plus one review period, with a safety buffer calibrated to the newsvendor critical fractile $\alpha_k = b_k/(b_k + h_k)$:
\begin{equation}\label{eq:out_level}
S_k(t) = (L_k + 1)\,\hat\mu_k(t) + z_{\alpha_k}\,\hat\sigma_k(t)\,\sqrt{L_k + 1},
\end{equation}
where $\hat\mu_k(t)$ and $\hat\sigma_k(t)$ are the forecast mean and standard deviation of echelon $k$'s incoming demand, and $z_{\alpha_k} = \Phi^{-1}(\alpha_k)$ is the corresponding safety factor. The order quantity is the gap between this target and the current inventory position $\mathrm{IP}_k(t) = I_k(t) + \sum_{\tau=1}^{L_k} P_k(\tau)$, truncated at zero to prevent negative orders:
\begin{equation}\label{eq:order}
O_k(t) = \max\!\bigl\{0,\; S_k(t) - \mathrm{IP}_k(t)\bigr\}.
\end{equation}

\paragraph{Inventory and cost.} After ordering, inventory is updated to reflect receipts from the pipeline and demand fulfillment:
\begin{equation}\label{eq:inventory}
I_k(t) = I_k(t-1) + R_k(t) - D_k(t),
\end{equation}
where $R_k(t)$ is the quantity received, equal to the order placed $L_k$ periods earlier. Positive inventory incurs holding cost at rate $h_k$; negative inventory (backorders) incurs penalty at rate $b_k$:
\begin{equation}\label{eq:cost}
C_k(t) = h_k \bigl[I_k(t)\bigr]^+ + b_k \bigl[-I_k(t)\bigr]^+.
\end{equation}
Total supply chain cost aggregates across all echelons and periods:
\begin{equation}\label{eq:total_cost}
\mathrm{TC} = \sum_{k=1}^{K} \sum_{t=1}^{T} C_k(t).
\end{equation}

\paragraph{Bullwhip amplification.} The cumulative bullwhip ratio through echelon $K$ is the product of the per-echelon variance ratios defined in Eq.~\eqref{eq:bwr}:
\begin{equation}\label{eq:cum_bwr}
\mathrm{BWR}_{\mathrm{cum}}(K) = \prod_{k=1}^{K} \mathrm{BWR}_k.
\end{equation}
For $k > 1$, the demand cascade (Eq.~\ref{eq:cascade}) implies $D_k = O_{k-1}$, so $\mathrm{BWR}_k = \mathrm{Var}(O_k)/\mathrm{Var}(O_{k-1})$. Since both numerator and denominator depend on the same demand realization, the cross-path variability of $\mathrm{BWR}_k$ is governed by their correlation. The following proposition formalizes this relationship.

\begin{proposition}[Variance ratio filtering]\label{prop:filtering}
Let $X_\omega = \mathrm{Var}(O_k \mid \omega)$ and $Y_\omega = \mathrm{Var}(O_{k-1} \mid \omega)$ denote the order variances at echelons $k$ and $k{-}1$ conditional on demand realization $\omega$, and let $\mathrm{BWR}_k(\omega) = X_\omega / Y_\omega$. If $X_\omega$ and $Y_\omega$ are positive random variables with finite first and second moments, then
\begin{equation}\label{eq:cv_ratio}
\mathrm{CV}\!\bigl(\mathrm{BWR}_k\bigr) \;\approx\; \sqrt{\mathrm{CV}(X)^2 + \mathrm{CV}(Y)^2 - 2\,\rho_{XY}\,\mathrm{CV}(X)\,\mathrm{CV}(Y)},
\end{equation}
where $\rho_{XY} = \mathrm{Corr}(X, Y)$ and the approximation holds to first order in $\mathrm{CV}(X)$ and $\mathrm{CV}(Y)$. In particular, $\mathrm{CV}(\mathrm{BWR}_k) \to 0$ as $\rho_{XY} \to 1$ whenever $\mathrm{CV}(X) \approx \mathrm{CV}(Y)$.
\end{proposition}

\noindent The proof, based on the multivariate delta method, is given in Appendix~\ref{app:proof_filtering}. Equation~\eqref{eq:cv_ratio} predicts that the cross-path variability of $\mathrm{BWR}_k$ collapses at echelons where the demand cascade induces high correlation between successive order variances. We validate this prediction empirically in Section~\ref{sec:exp_sim}.

More broadly, the multiplicative structure of Eq.~\eqref{eq:cum_bwr} has an important consequence: even moderate per-echelon ratios compound rapidly across tiers. Four echelons each amplifying by a factor of $3$ produce cumulative amplification of $3^4 = 81$, and heterogeneous lead times (e.g., in semiconductor supply chain shown in Table~\ref{tab:config}) produce amplification far beyond this symmetric case. This compounding is the primary reason why the single-echelon bounds of \citet{Chen2000} underestimate upstream variance exposure in multi-echelon chains, and why full-chain simulation is necessary, which the proposed software engine enables.

An immediate consequence of Proposition~\ref{prop:filtering} extends to the cumulative bullwhip ratio itself.

\begin{corollary}[Cumulative concentration]\label{cor:cumulative}
Let $\mathrm{BWR}_{\mathrm{cum}} = \prod_{k=1}^{K} \mathrm{BWR}_k$. Applying the multivariate delta method to the product, the cross-path variability of $\mathrm{BWR}_{\mathrm{cum}}$ satisfies
\begin{equation}\label{eq:cv_cum}
\mathrm{CV}\!\bigl(\mathrm{BWR}_{\mathrm{cum}}\bigr)^2 \;\approx\; \sum_{k=1}^{K} \mathrm{CV}(\mathrm{BWR}_k)^2 + 2 \!\!\sum_{1 \le i < j \le K}\!\! \rho_{ij}\,\mathrm{CV}(\mathrm{BWR}_i)\,\mathrm{CV}(\mathrm{BWR}_j),
\end{equation}
where $\rho_{ij} = \mathrm{Corr}(\mathrm{BWR}_i, \mathrm{BWR}_j)$. When the stochastic filtering of Proposition~\ref{prop:filtering} drives $\mathrm{CV}(\mathrm{BWR}_k) \to 0$ for all $k \ge k^*$, the sum is dominated by the first $k^*{-}1$ terms, and the \emph{uncertainty} in cumulative amplification is governed by the downstream tiers---precisely the tiers that contribute the least amplification.
\end{corollary}

\noindent The proof follows directly from the multivariate delta method applied to $g(Z_1,\ldots,Z_K) = \prod_k Z_k$; the full derivation is given in Appendix~\ref{app:proof_cumulative}. Corollary~\ref{cor:cumulative} carries a practical implication that is, at first glance, counterintuitive: the foundry tier (E3) contributes the largest per-echelon amplification ($\mathrm{BWR}_3 \approx 10$) yet contributes negligibly to the \emph{stochastic uncertainty} in cumulative BWR. We validate this prediction in Section~\ref{sec:exp_sim}, where the full covariance formula~\eqref{eq:cv_cum} predicts $\mathrm{CV}(\mathrm{BWR}_{\mathrm{cum}}) = 0.232$ against an empirical value of $0.236$ (1.8\% error across 5{,}000 paths).

\begin{table}[htbp]
\centering
\caption{Default semiconductor supply chain configuration. The critical fractile $\alpha_k = b_k/(b_k + h_k)$ increases upstream, reflecting higher safety stock requirements at tiers with longer pipelines.}\label{tab:config}
\begin{tabular}{clcccc}
\toprule
Echelon & Role & $L_k$ (wk) & $h_k$ & $b_k$ & $\alpha_k$ \\
\midrule
E1 & Distributor / OEM & 2 & 0.15 & 0.60 & 0.80 \\
E2 & Assembly \& Test (OSAT) & 4 & 0.12 & 0.50 & 0.81 \\
E3 & Foundry / Fab & 12 & 0.08 & 0.40 & 0.83 \\
E4 & Wafer / Material & 8 & 0.05 & 0.30 & 0.86 \\
\bottomrule
\end{tabular}
\end{table}

The cost ratios $b_k/h_k$ range from 4.0 at E1 to 6.0 at E4, corresponding to critical fractiles $\alpha_k \in [0.80, 0.86]$. These are consistent with the service level targets of 80–90\% commonly reported for multi-echelon semiconductor supply chains \citep{SilverPykeThomas2017,CachonTerwiesch2012}, where annual holding costs of 20–30\% of inventory value and high stockout penalties from customer contract losses produce $b/h$ ratios in the range 4–10 \citep{Zipkin2000}.

\subsection{Software architecture}\label{sec:architecture}

The \texttt{deepbullwhip} package is organized to follow the Strategy pattern \citep{GoF1994}: demand generation, ordering policy, cost computation, and forecasting are each encapsulated behind an abstract base class (ABC) with a single required method, so that any component can be replaced without modifying the simulation engine. The \texttt{DemandGenerator} ABC requires a method \texttt{generate(T, seed)} that returns a 1D demand array; the \texttt{OrderingPolicy} ABC requires \texttt{compute\_order(ip, fm, fs)} that returns a scalar order quantity; the \texttt{CostFunction} ABC requires \texttt{compute(inventory)} that returns a scalar cost; and the \texttt{Forecaster} ABC requires \texttt{forecast(history, steps)} that returns a mean--standard-deviation pair. Each ABC is accompanied by multiple concrete model registries described in Section~\ref{sec:benchmark}.

\begin{table}[htbp]
\centering
\caption{Package module structure.}\label{tab:modules}
\resizebox{\linewidth}{!}{
\begin{tabular}{ll}
\toprule
Module & Responsibility \\
\midrule
\texttt{demand/} & Demand generators (AR(1), ARMA, Beer Game, Replay) \\
\texttt{policy/} & Ordering policies (OUT, POUT, Smoothing, Constant) \\
\texttt{cost/} & Cost functions (Newsvendor, Perishable) \\
\texttt{forecast/} & Forecasters (Na\"ive, Moving Average, Exp.\ Smoothing, DeepAR) \\
\texttt{chain/} & Serial and vectorized simulation engines \\
\texttt{metrics/} & BWR, NSAmp, Fill Rate, Total Cost, Chen bound \\
\texttt{datasets/} & Data loaders (WSTS, M5, Beer Game, synthetic) \\
\texttt{benchmark/} & Registry, BenchmarkRunner, report export \\
\texttt{diagnostics/} & Plot functions, network diagrams \\
\bottomrule
\end{tabular}
}
\end{table}

The serial engine (\texttt{SerialSupplyChain}) iterates over $T$ periods. At each period, echelon $k = 1$ receives the customer demand $D(t)$ and exogenous forecasts $(\hat\mu_t, \hat\sigma_t)$; it computes its order via Eqs.~\eqref{eq:out_level}--\eqref{eq:order}, appends the order to its pipeline, satisfies demand from inventory via Eq.~\eqref{eq:inventory}, and evaluates cost via Eq.~\eqref{eq:cost}. The order placed by echelon $k$ becomes the demand for echelon $k+1$ through the cascade in Eq.~\eqref{eq:cascade}. Upstream echelons ($k > 1$) estimate their own forecasts as rolling statistics of the orders received over the most recent 8 periods. This cascade continues through all $K$ echelons before advancing to $t+1$. Figure~\ref{fig:architecture} illustrates the simulation engine architecture.

\begin{figure}
    \centering
    \newcommand{\method}[1]{{\ttfamily\small\color{blue!70!black}#1}}
    \resizebox{\textwidth}{!}{
    \begin{tikzpicture}[
        font=\sffamily\small,
        >=Stealth,
        engine_style/.style={
            draw=black, thick, fill=blue!8, text width=4.4cm, text centered, 
            minimum height=1.4cm, rounded corners=3pt
        },
        abc_style/.style={
            draw=black, thick, fill=white, text width=3.8cm, 
            minimum height=1.1cm, rounded corners=2pt, align=center
        },
        arrow_label/.style={
            font=\small, color=black, fill=white, inner sep=2pt
        },
    ]
    
        \node[engine_style] (engine) {
            \textbf{SerialSupplyChain}\\[1pt]
            {\small $t = 1,\ldots,T$\;\; $k = 1,\ldots,K$}
        };
    
        \node[abc_style, above=1.4cm of engine] (demand) {
            \textbf{DemandGenerator}\\[1pt]
            \method{generate(T, seed)}
        };
    
        \node[abc_style, right=2.2cm of engine] (policy) {
            \textbf{OrderingPolicy}\\[1pt]
            \method{compute\_order(ip, $\hat\mu$, $\hat\sigma$)}
        };
    
        \node[abc_style, below=1.4cm of engine] (forecaster) {
            \textbf{Forecaster}\\[1pt]
            \method{forecast(hist, steps)}
        };
    
        \node[abc_style, left=2.2cm of engine] (cost) {
            \textbf{CostFunction}\\[1pt]
            \method{compute(inventory)}
        };
    
        \draw[->, black, thick] (demand) -- (engine) 
            node[arrow_label, midway, right, xshift=2pt] {$D_k(t)$};
        
        \draw[->, black, thick] (forecaster) -- (engine) 
            node[arrow_label, midway, right, xshift=2pt] 
            {$\hat\mu_k,\; \hat\sigma_k$};
        
        \draw[->, black, thick] 
            ([yshift=2pt]engine.east) -- ([yshift=2pt]policy.west) 
            node[arrow_label, midway, above, yshift=1pt] 
            {$\mathrm{IP}_k,\; \hat\mu_k,\; \hat\sigma_k$};
        \draw[<-, black, thick] 
            ([yshift=-4pt]engine.east) -- ([yshift=-4pt]policy.west) 
            node[arrow_label, midway, below, yshift=-1pt] {$O_k(t)$};
        
        \draw[->, black, thick] 
            ([yshift=2pt]engine.west) -- ([yshift=2pt]cost.east) 
            node[arrow_label, midway, above, yshift=1pt] {$I_k(t)$};
        \draw[<-, black, thick] 
            ([yshift=-4pt]engine.west) -- ([yshift=-4pt]cost.east) 
            node[arrow_label, midway, below, yshift=-1pt] {$C_k(t)$};
    
        \draw[->, blue!70!black, line width=1pt, dashed, 
              rounded corners=5pt] 
            (policy.north) |- (demand.east)
            node[pos=0.25, right, font=\small\itshape, 
                 color=blue!70!black, fill=white, inner sep=2pt] 
            {$D_{k+1}(t) = O_k(t)$};
    
    \end{tikzpicture}
    }
    \caption{Architecture of the \texttt{deepbullwhip} simulation engine. At each period $t$, the engine receives demand from the \texttt{DemandGenerator}, obtains forecasts from the \texttt{Forecaster}, queries the \texttt{OrderingPolicy} for an order quantity $O_k(t)$, and evaluates inventory cost via the \texttt{CostFunction}. The dashed arrow shows the demand cascade: orders at echelon $k$ become demand at echelon $k{+}1$.}\label{fig:architecture}
\end{figure}
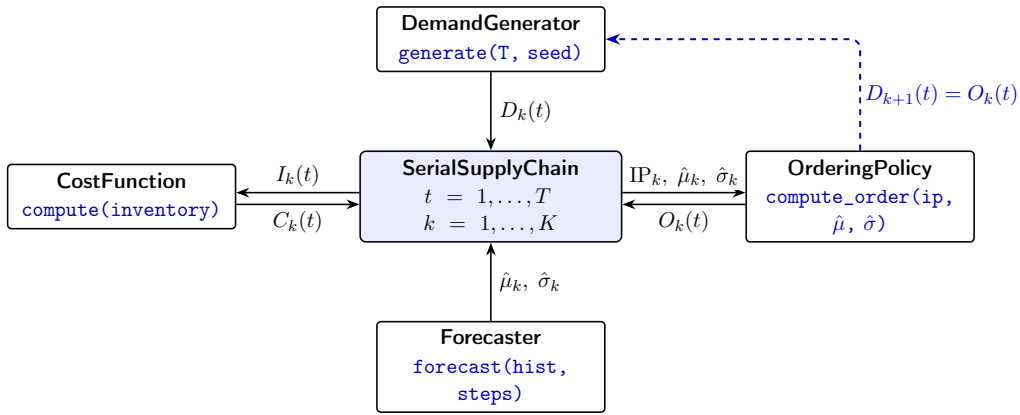 

\subsection{Vectorized Monte Carlo engine}\label{sec:vectorized}

Monte Carlo simulation is essential for characterizing the distributional properties of bullwhip metrics across demand realizations, but a na\"ive Python implementation that loops over $N$ paths, $K$ echelons, and $T$ periods is prohibitively slow for large-scale experiments. The vectorized engine addresses this by pre-allocating three-dimensional NumPy arrays of shape $(N, K, T)$ for orders, inventory, and costs, and a pipeline tensor of shape $(N, K, L_{\max})$ where $L_{\max} = \max_k L_k$. At each time step $t$, the engine broadcasts the OUT policy computation across the $N$ and $K$ dimensions simultaneously:
\begin{equation}\label{eq:vec_level}
\mathbf{S}(t) = (\mathbf{L} + \mathbf{1}) \odot \hat{\boldsymbol{\mu}}(t) + \mathbf{z}_\alpha \odot \hat{\boldsymbol{\sigma}}(t) \odot \sqrt{\mathbf{L} + \mathbf{1}},
\end{equation}
\begin{equation}\label{eq:vec_order}
\mathbf{O}(t) = \max\!\bigl\{\mathbf{0},\; \mathbf{S}(t) - \mathbf{IP}(t)\bigr\},
\end{equation}
where $\mathbf{L}$ and $\mathbf{z}_\alpha$ are vectors of length $K$, $\hat{\boldsymbol{\mu}}(t)$ and $\hat{\boldsymbol{\sigma}}(t)$ are $(N \times K)$ matrices, and $\odot$ denotes element-wise multiplication with broadcasting. A circular buffer indexed by a pointer array of length $K$ provides $O(1)$ pipeline receipt operations, replacing the $O(L_k)$ list manipulation of the serial engine. The time loop over $T$ remains sequential because the demand cascade in Eq.~\eqref{eq:cascade} introduces a within-period dependency, but all $N$ paths and the policy computations within each time step are fully parallelized. This design achieves 50 to 90 times speedup over the serial engine across the tested range of $N \in [10, 5000]$, $T \in [52, 5200]$, and $K \in [2, 16]$, with both engines scaling as $O(NKT)$ in wall time. The largest stress test processed $5{,}000 \times 520 \times 8 = 20.8 \times 10^6$ simulation cells in 6.8 seconds with 476\,MB peak memory on a single CPU core. Full scalability benchmarks are reported in Section~\ref{sec:exp_scalability}.


\section{Benchmarking framework}\label{sec:benchmark}

Every new bullwhip study effectively starts from scratch: researchers
select their own metrics, code their own demand processes, and compare
against ad-hoc baselines, making cross-study comparison unreliable
\citep{Brauch2024,Kourentzes2025}.  The forecasting community solved an
analogous problem through the M-competitions
\citep{Makridakis2020m4,Makridakis2022}, which established shared
datasets and evaluation protocols that enabled cumulative learning
across hundreds of methods.  No equivalent infrastructure exists for the
bullwhip effect. The benchmarking framework in \texttt{deepbullwhip} addresses this gap
through three components: a model registry for extensibility
(Section~\ref{sec:registry}), a curated catalog of models, metrics, and
datasets (Section~\ref{sec:catalog}), and a benchmark runner for
standardized evaluation (Section~\ref{sec:runner}).

\subsection{Model registry}\label{sec:registry}

Researchers should be able to introduce a new ordering policy,
forecaster, or metric and immediately evaluate it against every baseline
in the catalog without touching the simulation engine or the runner.
We achieve this through a \emph{decorator-based model registry}, a
pattern widely adopted in machine learning toolkits such as OpenMMLab
\citep{Chen2019mmdetection} and Detectron2 \citep{Wu2019detectron2}. The registry maintains a global dictionary mapping string identifiers to
class constructors across five categories: demand generators, ordering
policies, cost functions, forecasters, and metrics.  Registration
happens at import time via a Python class decorator; retrieval happens at
run time by name (Algorithm~\ref{alg:registry}).  A researcher who
develops a new ordering policy need only subclass the
\texttt{OrderingPolicy} abstract base class, implement
\texttt{compute\_order}, and apply the \texttt{@register} decorator.
The policy is then available to the runner by name, with no changes to
any other file.  A code example appears in Appendix~\ref{app:code}.
The pseudocode for the registry mechanism is given in Algorithm~\ref{alg:registry} (Appendix~\ref{app:algorithms}). This architecture follows the Open/Closed Principle
\citep{Martin2003agile}: the framework is open to extension so that community contributions can be integrated with a single
decorated class and no changes to the core package.

\subsection{Catalog of models, metrics, and datasets}\label{sec:catalog}

Table~\ref{tab:catalog} summarizes the components that ship with
\texttt{deepbullwhip}.  The catalog spans four ordering policies
(from full demand chasing under OUT to complete smoothing under POUT),
three classical forecasters that mirror the assumptions in the
analytical bounds of \citet{Chen2000} and \citet{Zhang2004} plus the DeepAR probabilistic forecaster \citep{Salinas2020deepar}, six
bullwhip and cost metrics reflecting the recent consensus that
multi-metric evaluation is essential
\citep{Brauch2024,Kourentzes2025}, and four demand generators covering
both synthetic processes and historical replay.  Default parameter
values, descriptions, and literature references for each entry are
distributed as a structured JSON file alongside the package. The dataset module bundles a 60-month sample of WSTS semiconductor
billings \citep{wsts2025} and provides a loader for the M5 Walmart
dataset \citep{Makridakis2022} that downloads it on first use (though requiring an internet connection and account creation).
Three predefined chain configurations are registered:
\texttt{semiconductor\_4tier} (Table~\ref{tab:config}),
\texttt{beer\_game} \citep{Sterman1989}, and
\texttt{consumer\_2tier}.

\begin{table}[htbp]
\centering
\caption{Registered components in the benchmarking catalog.}\label{tab:catalog}
\resizebox{\linewidth}{!}{
\begin{tabular}{llp{7cm}}
\toprule
Category & Name & Description \\
\midrule
\multirow{4}{*}{Policies}
  & \texttt{order\_up\_to} & Standard OUT, Eq.~\eqref{eq:order} \citep{Chen2000} \\
  & \texttt{proportional\_out} & POUT with smoothing $\alpha \in (0,1]$ \citep{Disney2003} \\
  & \texttt{smoothing\_out} & OUT with exponential order smoothing \\
  & \texttt{constant\_order} & Fixed order quantity each period \\
\midrule
\multirow{4}{*}{Forecasters}
  & \texttt{naive} & Constant forecast (sample mean and std) \\
  & \texttt{moving\_average} & Rolling window of configurable length $p$ \\
  & \texttt{exponential\_smoothing} & Single exp.\ smoothing with parameter $\alpha$ \\
  & \texttt{deepar} & DeepAR probabilistic forecaster \citep{Salinas2020deepar} \\
\midrule
\multirow{6}{*}{Metrics}
  & \texttt{BWR} & Bullwhip ratio, Eq.~\eqref{eq:bwr} \\
  & \texttt{CUM\_BWR} & Cumulative BWR, Eq.~\eqref{eq:cum_bwr} \\
  & \texttt{NSAmp} & Net Stock Amplification \citep{Disney2003} \\
  & \texttt{FILL\_RATE} & Fraction of periods with $I_k(t) \geq 0$ \\
  & \texttt{TC} & Total cost, Eq.~\eqref{eq:total_cost} \\
  & \texttt{ChenLowerBound} & Analytical bound of \citet{Chen2000} \\
\midrule
\multirow{4}{*}{Demand}
  & \texttt{semiconductor\_ar1} & AR(1) with shock, Eq.~\eqref{eq:demand} \\
  & \texttt{beer\_game} & Step function: 4 then 8 \citep{Sterman1989} \\
  & \texttt{arma} & Configurable ARMA($p$,$q$) process \\
  & \texttt{replay} & Replay historical data from array or CSV \\
\midrule
\multirow{2}{*}{Costs}
  & \texttt{newsvendor} & $h[I]^+ + b[-I]^+$, Eq.~\eqref{eq:cost} \\
  & \texttt{perishable} & Newsvendor + obsolescence penalty \\
\bottomrule
\end{tabular}
}
\end{table}

\subsection{Benchmark runner}\label{sec:runner}

The \texttt{BenchmarkRunner} class automates standardized evaluation.
It accepts a chain configuration~$\mathcal{C}$, a demand
generator~$\mathcal{G}$, a time horizon~$T$, a number of Monte Carlo
paths~$N$, and a random seed.  Given sets of policy
names~$\mathcal{P}$, forecaster names~$\mathcal{F}$, and metric
names~$\mathcal{M}$, it evaluates every $(p, f)$ combination against
every metric~$m$ at every echelon~$k$, returning a structured table
with columns (policy, forecaster, echelon, metric, value).
Algorithm~\ref{alg:runner} formalizes the procedure.  Policies and
forecasters may be specified with custom parameters; default values are
used otherwise.  Output is directly exportable to \LaTeX{} and CSV
through the companion \texttt{report} module.

\begin{algorithm}[htbp]
\caption{Benchmark evaluation procedure}\label{alg:runner}
\begin{algorithmic}[1]
\Require Chain config $\mathcal{C}$, demand generator $\mathcal{G}$, horizon $T$, paths $N$, seed $s$
\Require Policy set $\mathcal{P}$, forecaster set $\mathcal{F}$, metric set $\mathcal{M}$
\Ensure Results table $\mathcal{R}$ with columns (policy, forecaster, echelon, metric, value)
\State $\mathbf{D} \gets \mathcal{G}.\text{generate\_batch}(T, N, s)$ \Comment{$(N \times T)$ demand matrix}
\For{each policy $p \in \mathcal{P}$}
  \For{each forecaster $f \in \mathcal{F}$}
    \State Build chain from $\mathcal{C}$ with policy $p$
    \State $(\hat{\boldsymbol{\mu}}, \hat{\boldsymbol{\sigma}}) \gets f.\text{generate\_forecasts}(\mathbf{D})$ \Comment{$(N \times T)$ forecast arrays}
    \State $\text{result} \gets \text{chain.simulate}(\mathbf{D}, \hat{\boldsymbol{\mu}}, \hat{\boldsymbol{\sigma}})$
    \For{each echelon $k = 1, \ldots, K$}
      \For{each metric $m \in \mathcal{M}$}
        \State Append $(p, f, k, m, m.\text{compute}(\text{result}, \mathbf{D}, k))$ to $\mathcal{R}$
      \EndFor
    \EndFor
  \EndFor
\EndFor
\State \Return $\mathcal{R}$
\end{algorithmic}
\end{algorithm}

Finally, it is worth noting that because the runner discovers components through the registry at
execution time, adding a new policy, forecaster, or metric requires no
modification to the runner itself. This is a key advantage of the registry-based architecture as part of our design choices, as it allows for easy extension of the framework without modifying the core components.

\section{Numerical experiments}\label{sec:experiments}

We present five sets of experiments organized into three subsections. Section~\ref{sec:exp_sim} reports three simulation experiments that validate the engine and characterize the structural properties of multi-echelon bullwhip dynamics: single-path propagation, Monte Carlo stochastic filtering, and lead time sensitivity. Section~\ref{sec:exp_scalability} benchmarks computational scalability. Section~\ref{sec:exp_benchmark} demonstrates the benchmarking framework through four comparative studies using the \texttt{BenchmarkRunner}. All experiments use the default semiconductor chain configuration from Table~\ref{tab:config}. Jupyter notebooks reproducing each experiment are included in the package repository. Supplementary figures (order streams, inventory trajectories, cost decomposition, echelon details) appear in Appendix~\ref{app:diagnostics}.

\subsection{Simulation experiments}\label{sec:exp_sim}

The default experiments included in the package are as follows:

\subsubsection{Experiment 1: Bullwhip propagation} 

A single demand realization ($T = 156$ weeks, seed $= 42$) is simulated with a constant forecast to isolate the OUT policy's amplification mechanism. Table~\ref{tab:exp1} reports the results. The bullwhip ratio increases from 1.12 at E1 to 37.56 at E4, producing cumulative $\text{BWR}_{\text{cum}} = 838$. The foundry tier (E3, $L_3 = 12$) contributes the largest single-echelon ratio ($\text{BWR}_3 = 10.81$), exceeding the Chen lower bound of 1.57 for $p = 52$ because E3 faces already-amplified orders from E2 rather than end-customer demand. Fill rates at E2 (69.2\%) and E3 (67.9\%) fall below their newsvendor-optimal targets, confirming that constant forecasts are inadequate for long-pipeline chains. The summary dashboard (Fig.~\ref{fig:dashboard}) shows three notable features: the structural shock at $t = 104$ produces E3 and E4 order spikes several times larger than the demand perturbation; E4 experiences an inventory drawdown of approximately 1,000 units representing substantial working capital exposure; and the BWR profile exhibits the multiplicative cascade pattern described in Eq.~\eqref{eq:cum_bwr}.

\begin{table}
\centering
\caption{Experiment 1: Single-path results ($T = 156$, seed $= 42$, constant forecast).}\label{tab:exp1}
\begin{tabular}{clrrrr}
\toprule
Echelon & Role & $\text{BWR}_k$ & Fill Rate & Cost & $\text{BWR}_{\text{cum}}$ \\
\midrule
E1 & Distributor & 1.12  & 85.3\% & 344   & 1.1 \\
E2 & Assembly    & 1.84  & 69.2\% & 265   & 2.1 \\
E3 & Foundry     & 10.81 & 67.9\% & 1,196 & 22.3 \\
E4 & Wafer       & 37.56 & 85.9\% & 3,373 & 838.3 \\
\midrule
\multicolumn{2}{l}{\textit{Total}} & --- & --- & 5,178 & 838.3 \\
\bottomrule
\end{tabular}
\end{table}

\begin{figure}
\centering
\includegraphics[width=\textwidth]{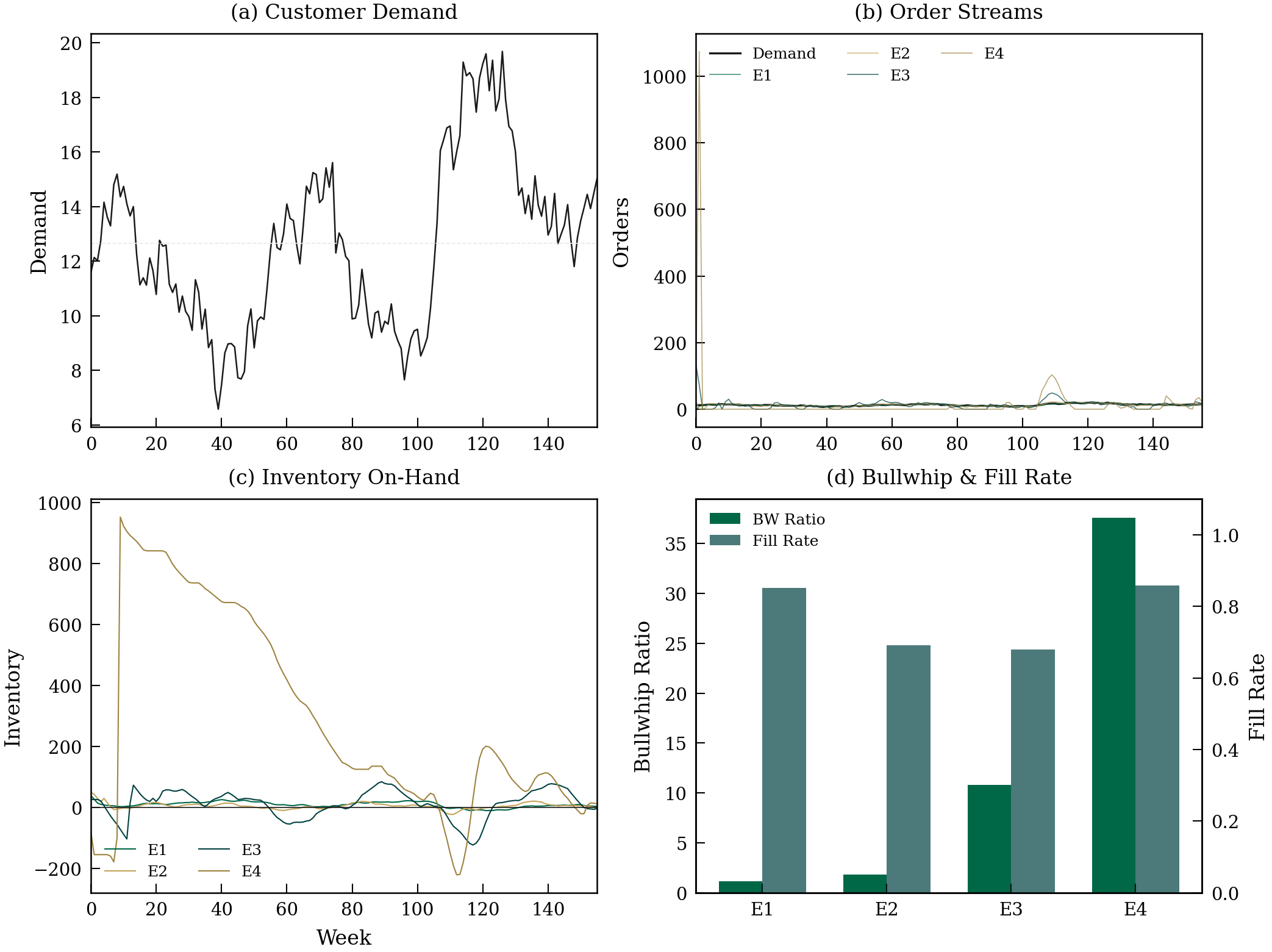}
\caption{Experiment 1 summary dashboard. (a)~AR(1) demand with structural shock at $t = 104$. (b)~Order streams showing upstream spike magnification. (c)~Inventory trajectories with ${\sim}1{,}000$-unit drawdown at E4. (d)~Per-echelon BWR and fill rate.}\label{fig:dashboard}
\end{figure}

To validate the simulator against analytical results, we ran a rigorous single-echelon experiment with i.i.d.\ Normal demand ($\mu = 100$, $\sigma = 10$), moving-average forecasting of window~$p$, and the OUT policy across $N = 2{,}000$ paths and $T = 520$ periods. Table~\ref{tab:chen_validation} compares the simulated BWR to the Chen lower bound $\text{BWR}_k \geq 1 + 2(L_k{+}1)/p + 2(L_k{+}1)^2/p^2$ (where $L_k{+}1$ is the review lead time). The simulated BWR matches the bound to within 0.14\% across all test cases, confirming that the simulator correctly implements the OUT policy dynamics under the bound's own assumptions. The small positive deviations are consistent with the finite-sample bias of the sample variance estimator.

\begin{table}[htbp]
\centering
\caption{Rigorous validation against the Chen et al.\ (2000) lower bound under i.i.d.\ demand ($N = 2{,}000$, $T = 520$). The bound is tight for i.i.d.\ demand with MA($p$) forecasting.}\label{tab:chen_validation}
\begin{tabular}{rrrrrr}
\toprule
$L_k$ & $p$ & Chen LB & Simulated & Error \\
\midrule
2  & 10  & 1.780 & 1.781 & $+$0.04\% \\
4  & 10  & 2.500 & 2.502 & $+$0.08\% \\
4  & 20  & 1.625 & 1.625 & $-$0.00\% \\
8  & 20  & 2.305 & 2.306 & $+$0.04\% \\
8  & 52  & 1.406 & 1.406 & $+$0.02\% \\
12 & 52  & 1.625 & 1.625 & $+$0.01\% \\
2  & 52  & 1.122 & 1.122 & $+$0.00\% \\
12 & 10  & 6.980 & 6.990 & $+$0.14\% \\
\bottomrule
\end{tabular}
\end{table}

For the multi-echelon semiconductor chain, the single-path simulation (Table~\ref{tab:exp1}) produces BWR that exceeds the Chen bound by increasing multiples at upstream echelons (1.6$\times$ at E2, 6.9$\times$ at E3, 27.7$\times$ at E4), because the Chen bound applies to a single echelon facing exogenous demand, whereas upstream echelons face already-amplified orders from downstream. This growing divergence is the multi-echelon compounding effect that single-tier bounds do not capture and that motivates full-chain simulation.

Note on the difference between the single-path cumulative BWR of 838 (Table~\ref{tab:exp1}) and the Monte Carlo mean of 427 reported later in the benchmarking experiments (Table~\ref{tab:bench_policy}): Experiment~1 uses a \emph{constant} forecast computed from the full demand series (which leaks future information), while the benchmark experiments use the \texttt{naive} forecaster that estimates demand statistics from the rolling history up to each period. The rolling estimator reduces over-ordering at upstream tiers, producing lower cumulative BWR. Additionally, the Monte Carlo mean averages over 1{,}000 demand realizations, whereas 838 is a single-path value. The structural shock at $t = 104$ makes single-path BWR particularly sensitive to the demand realization.

\subsubsection{Experiment 2: Stochastic filtering} 

Running $N = 1{,}000$ independent demand paths through the vectorized engine reveals an asymmetry in BWR variability across echelons. While the delta-method approximation for the CV of a ratio is a standard statistical technique, its application to cross-path BWR variability in multi-echelon supply chains appears to be new. Table~\ref{tab:exp2} reports the Monte Carlo statistics and Fig.~\ref{fig:mc_bwr} shows the distributions. Note that the E1 mean BWR of 151.0 in Table~\ref{tab:exp2} differs markedly from the single-path value of 1.12 in Table~\ref{tab:exp1} because Experiment~2 uses a \emph{global} constant forecast (the mean of all $N \times T$ demand values), which is a poor per-path predictor when individual demand realizations vary substantially around the global mean. In Experiment~1, the constant forecast is computed from the single demand realization, providing a much closer fit. This sensitivity of E1 BWR to forecast quality illustrates why the distributor tier carries the highest stochastic risk. The coefficient of variation (CV) of BWR drops from 0.21 at E1 to 0.01 at E3, indicating that the foundry's BWR is nearly deterministic regardless of demand realization.

Proposition~\ref{prop:filtering} provides a formal explanation for this phenomenon. The key mechanism is that $\text{BWR}_k = \text{Var}(O_k)/\text{Var}(O_{k-1})$ is a ratio of two random variables (varying across demand realizations) that become increasingly correlated at upstream echelons because both are driven by the same amplified demand cascade. Table~\ref{tab:filtering_validation} validates the delta-method prediction from Proposition~\ref{prop:filtering} against the empirical CV: at E3, the correlation between $\text{Var}(O_3)$ and $\text{Var}(O_2)$ reaches $\rho = 0.997$, and the predicted $\text{CV}(\text{BWR}_3) = 0.011$ matches the empirical value to four decimal places. The practical consequence is that upstream capacity planning at foundries and wafer suppliers can rely on deterministic scenario analysis parameterized by chain configuration, while the distributor tier requires Monte Carlo simulation to characterize its path-dependent risk.

To validate Corollary~\ref{cor:cumulative}, we computed the cumulative $\mathrm{BWR}_\mathrm{cum}$ across 5{,}000 independent paths. The empirical $\mathrm{CV}(\mathrm{BWR}_{\mathrm{cum}}) = 0.236$. A na\"ive independence approximation ($\sqrt{\sum \mathrm{CV}_k^2} = 0.281$) overpredicts by 19\% because adjacent echelons exhibit non-negligible cross-correlations, most notably $\rho_{34} = -0.74$ between the foundry and wafer tiers. Using the full covariance formula in Eq.~\eqref{eq:cv_cum} with empirically estimated $\rho_{ij}$ yields a prediction of 0.232, matching the empirical value within 1.8\%. The negative correlation between E3 and E4 arises because the demand cascade mechanically links their order variances: a demand realization that inflates the foundry's amplification tends to compress the wafer tier's ratio, and vice versa. 

Additionally, the total supply chain cost exhibits $\mathrm{CV}(\mathrm{TC}) = 0.079$, well below $\mathrm{CV}(\mathrm{BWR}_{\mathrm{E1}}) = 0.217$. This cost concentration effect arises because the majority of cost is incurred at the upstream tiers (E3 and E4 account for 88\% of total cost in Table~\ref{tab:exp1}), where the stochastic filtering phenomenon makes BWR near-deterministic. Consequently, the cross-path variability of total cost is far lower than the distributor-level bullwhip variability would suggest.

\begin{table}[htbp]
\centering
\caption{Validation of the stochastic filtering prediction (Proposition~\ref{prop:filtering}). $\rho$ is the cross-path correlation between $\text{Var}(O_k)$ and $\text{Var}(O_{k-1})$.}\label{tab:filtering_validation}
\resizebox{\linewidth}{!}{
\begin{tabular}{clrrrcc}
\toprule
Echelon & Role & $\text{CV}_X$ & $\text{CV}_Y$ & $\rho$ & Predicted CV & Empirical CV \\
\midrule
E2 & Assembly & 0.146 & 0.160 & 0.398 & 0.168 & 0.167 \\
E3 & Foundry  & 0.145 & 0.146 & 0.997 & 0.011 & 0.011 \\
E4 & Wafer    & 0.154 & 0.145 & 0.987 & 0.026 & 0.026 \\
\bottomrule
\end{tabular}
}
\end{table}

\begin{table}
\centering
\caption{Experiment 2: Monte Carlo BWR statistics ($N = 1{,}000$ paths).}\label{tab:exp2}
\begin{tabular}{clrrrr}
\toprule
Echelon & Role & Mean & Median & Std & CV \\
\midrule
E1 & Distributor & 151.0 & 149.6 & 32.4 & 0.21 \\
E2 & Assembly    & 18.0  & 17.9  & 3.0  & 0.17 \\
E3 & Foundry     & 8.1   & 8.1   & 0.1  & 0.01 \\
E4 & Wafer       & 22.2  & 22.2  & 0.6  & 0.03 \\
\bottomrule
\end{tabular}
\end{table}

\begin{figure}
\centering
\includegraphics[width=1.1\textwidth]{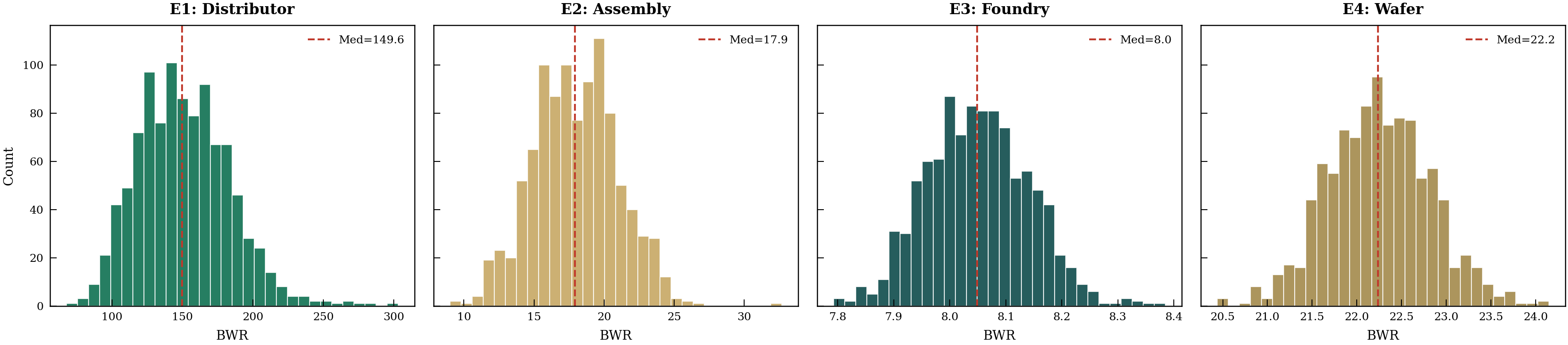}
\caption{BWR distributions across 1,000 Monte Carlo paths. E1 shows high stochastic variability; E3 converges to a narrow band around 8.05, illustrating the stochastic filtering phenomenon.}\label{fig:mc_bwr}
\end{figure}

\subsubsection{Experiment 3: Lead time and cost sensitivity} 

Three lead time scenarios---Short ($L = [1, 2, 4, 2]$), Baseline ($[2, 4, 12, 8]$), and Long ($[4, 8, 20, 12]$)---produce cumulative BWR spanning three orders of magnitude: 6.8, 838, and $1.8 \times 10^5$ (Fig.~\ref{fig:leadtime}). This super-exponential sensitivity has a direct strategic implication: reducing foundry lead time from 12 to 8 weeks (33\%) decreases cumulative BWR from 838 to approximately 90 (an order-of-magnitude improvement), whereas eliminating bullwhip entirely at E1 reduces it from 838 to 748 (11\%). We also vary the cost ratio $b/h$ from 2 to 10 under normalized costs ($h + b = 1$); higher ratios shift the cost profile from frequent holding costs to rare but extreme stockout events (Fig.~\ref{fig:cost_ratio} in Appendix~\ref{app:sensitivity}).

\begin{figure}
\centering
\includegraphics[width=0.7\textwidth]{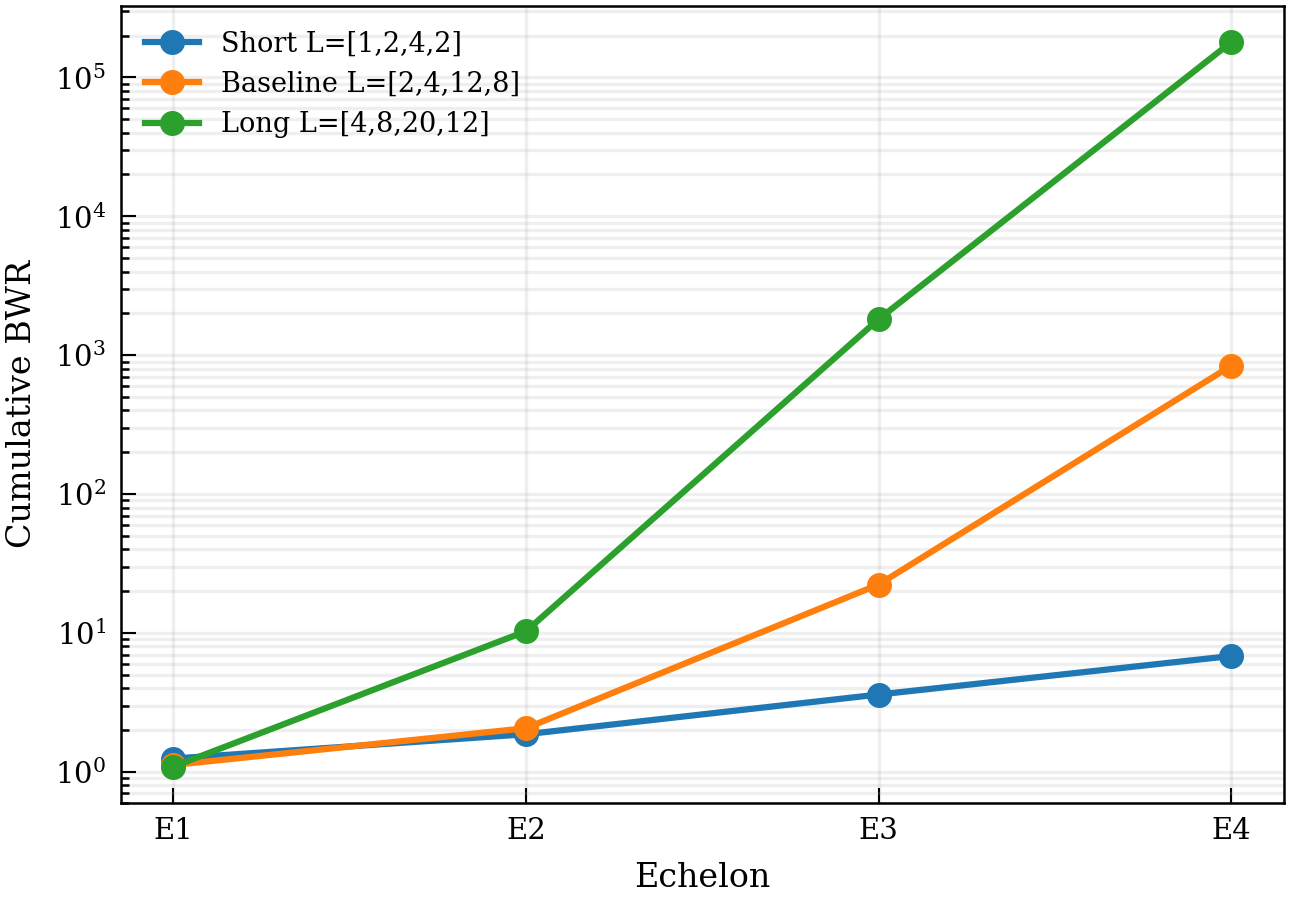}
\caption{Cumulative BWR vs.\ echelon for three lead time scenarios (logarithmic $y$-axis).}\label{fig:leadtime}
\end{figure}

\subsection{Computational scalability}\label{sec:exp_scalability}

The second set of experiments are scalability experiments that ablate the number of Monte Carlo paths $N$, the time horizon $T$, and the chain depth $K$. At the largest scale, it could simulate a realistic supply chain with 5,000 paths, 520 time periods, and 8 echelons in 6.8 seconds with 476\,MB peak memory on a single CPU core. In that regard, the vectorized engine maintains 50 to 90 times speedup over the serial engine across the tested ranges of Monte Carlo paths $N \in [10, 5{,}000]$, time horizon $T \in [52, 5{,}200]$ weeks, and chain depth $K \in [2, 16]$ echelons (Fig.~\ref{fig:scalability}). Both engines scale linearly in each dimension, confirming $O(NKT)$ complexity. Detailed timing tables are provided in Appendix~\ref{app:scalability}.

\begin{figure}
\centering
\includegraphics[width=1.1\textwidth]{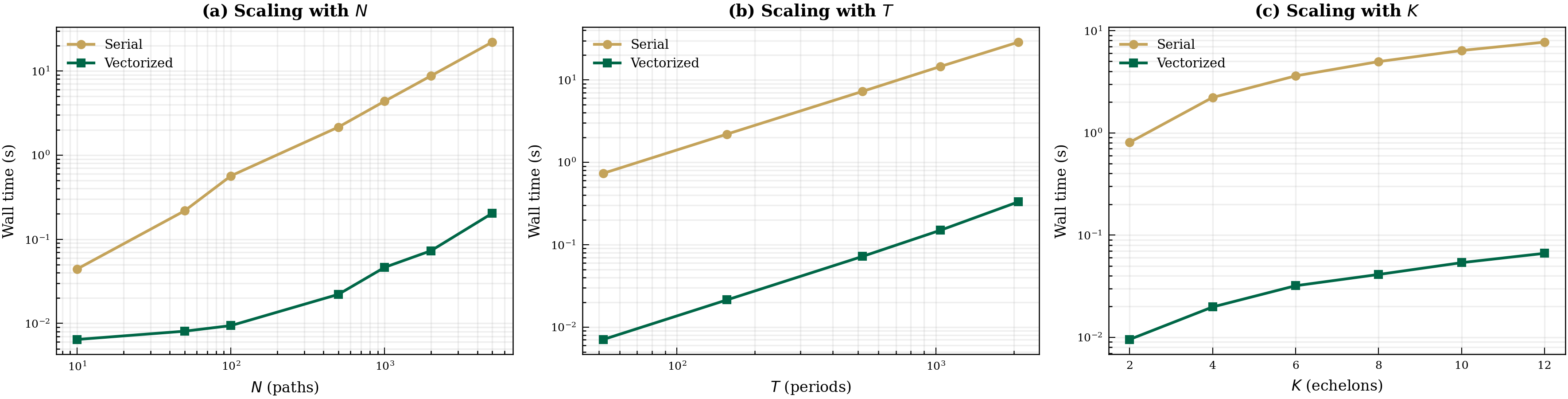}
\caption{Computational scalability: serial vs.\ vectorized engine across (a)~paths $N$, (b)~horizon $T$, and (c)~echelons $K$.}\label{fig:scalability}
\end{figure}

\subsection{Benchmark experiments}\label{sec:exp_benchmark}

The following four experiments demonstrate the benchmarking framework from Section~\ref{sec:benchmark}. Each is executed through a single call to \texttt{BenchmarkRunner.run()} and produces the standardized DataFrame described in Section~\ref{sec:runner}.

\subsubsection{Experiment 5a: Policy comparison.} 

Table~\ref{tab:bench_policy} compares four ordering policies on the semiconductor chain with $N = 1{,}000$ paths. The standard OUT policy produces cumulative BWR of 427 and total cost of 1,743 at E4. POUT with $\alpha = 0.3$ achieves a 96\% reduction in cumulative BWR (to 16.3) but at a heavy cost: fill rate collapses to near zero at E1 and total cost at E1 rises from 293 to 2,377, dominated by backorder penalties. Smoothing OUT reduces BWR more moderately (to 258) while maintaining comparable fill rates and even improving E4 fill rate to 88\%, but it produces dramatically higher NSAmp at E4 (12,952 vs.\ 3,189), meaning order smoothing shifts variance from orders to inventory, consistent with the theoretical analysis of \citet{Disney2003}. The Constant Order policy eliminates bullwhip by construction but produces near-zero fill rates at E3 and E4, confirming that ignoring demand signals is not a viable strategy. These results illustrate why multi-metric evaluation is essential: POUT appears superior on BWR alone but inferior on fill rate and upstream cost, and no single metric captures the full tradeoff (Fig.~\ref{fig:policy}).

\begin{table}
\centering
\caption{Experiment 5a: Policy comparison at E4 (Wafer Supplier), $N = 1{,}000$ paths.}\label{tab:bench_policy}
\begin{tabular}{lrrrr}
\toprule
Policy & Cum.\ BWR & NSAmp & Fill Rate & Total Cost \\
\midrule
OUT & 426.9 & 3,189 & 80\% & 1,743 \\
POUT ($\alpha{=}0.3$) & 16.3 & 407 & 50\% & 1,458 \\
Smoothing OUT & 257.5 & 12,952 & 88\% & 4,028 \\
Constant Order & 0.0 & 17 & 3\% & 2,262 \\
\bottomrule
\end{tabular}
\end{table}

\begin{figure}
\centering
\includegraphics[width=\textwidth]{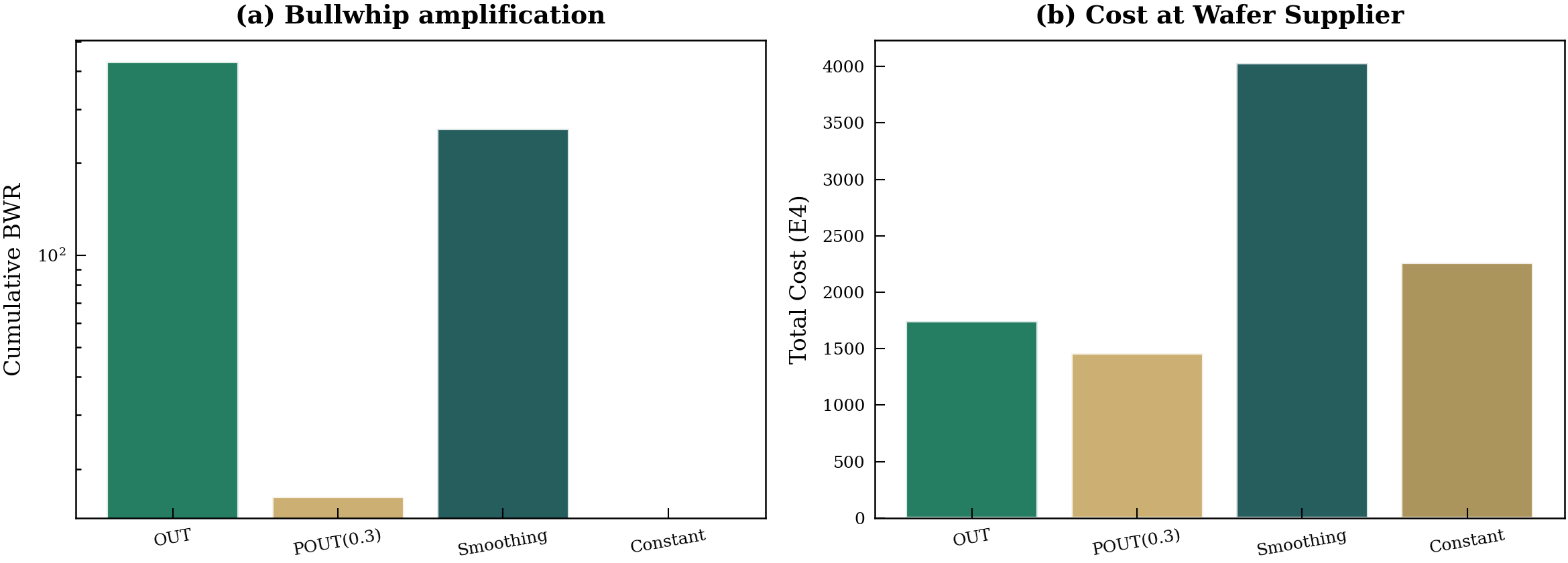}
\caption{Policy comparison at E4: (a)~cumulative BWR (logarithmic scale) and (b)~total cost.}\label{fig:policy}
\end{figure}

To characterize the full BWR--NSAmp tradeoff, we sweep the POUT smoothing parameter $\alpha$ from 0.1 to 1.0 in increments of 0.1 ($N = 500$ paths). Figure~\ref{fig:pareto} shows the resulting Pareto frontier at E4. As $\alpha$ decreases from 1.0 (standard OUT), cumulative BWR drops from 427 to 3.1 at $\alpha = 0.1$, but the tradeoffs are non-trivial: NSAmp first decreases from 3,189 to a minimum of 310 at $\alpha = 0.2$, then increases again for lower $\alpha$, revealing a non-monotonic inventory variance response. Fill rate degrades sharply below $\alpha = 0.4$ (from 61\% to 15\% at $\alpha = 0.1$), and total cost is minimized around $\alpha = 0.5$. This frontier quantifies the managerial tradeoff space and demonstrates that the benchmarking framework can identify optimal operating regions that discrete policy comparisons would miss.

\begin{figure}
\centering
\includegraphics[width=\textwidth]{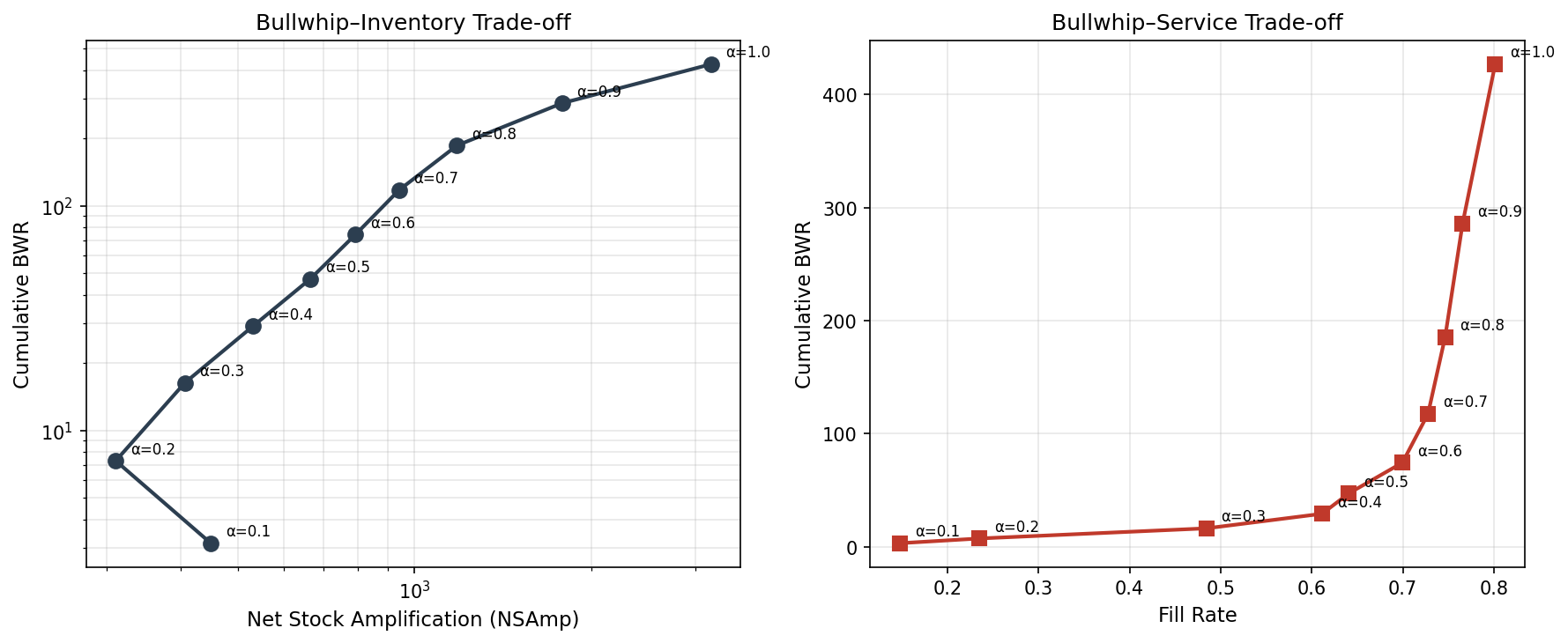}
\caption{POUT smoothing parameter tradeoff at E4: (a)~cumulative BWR vs.\ NSAmp (log-log), and (b)~cumulative BWR vs.\ fill rate. Labels indicate $\alpha$ values.}\label{fig:pareto}
\end{figure}

\subsubsection{Experiment 5a$'$: Cost asymmetry--policy interaction} 

The results above use a fixed cost configuration (Table~\ref{tab:config}). To test whether policy rankings are robust to cost structure, we evaluate three policies under normalized costs ($h + b = 1$) at four backorder-to-holding ratios $b/h \in \{2, 5, 10, 20\}$ ($N = 500$). Table~\ref{tab:cost_policy} reports total cost at E4. At low asymmetry ($b/h = 2$), POUT dominates with the lowest cost (9,423) because its dampened orders reduce inventory holding. As the backorder penalty rises, the aggressive demand-chasing of OUT becomes optimal ($b/h \geq 5$), and at $b/h = 20$, the Smoothing OUT policy overtakes POUT. This ranking reversal demonstrates that no single policy is universally optimal: the choice depends on the cost structure, and the benchmarking framework enables systematic exploration of this interaction.

\begin{table}
\centering
\caption{Total cost at E4 under three policies and four cost asymmetry levels ($N = 500$, normalized $h + b = 1$).}\label{tab:cost_policy}
\begin{tabular}{lrrrr}
\toprule
& \multicolumn{4}{c}{$b/h$ ratio} \\
\cmidrule(lr){2-5}
Policy & 2 & 5 & 10 & 20 \\
\midrule
OUT                   & 12{,}332 & 8{,}525 & 6{,}795 & 5{,}806 \\
POUT ($\alpha{=}0.5$) &  9{,}423 & 9{,}297 & 9{,}240 & 9{,}208 \\
Smoothing OUT         & 29{,}694 & 17{,}251 & 11{,}594 & 8{,}362 \\
\bottomrule
\end{tabular}
\end{table}

\subsubsection{Experiment 5b: Forecaster comparison} 

Table~\ref{tab:bench_fc} compares four forecasting methods under the OUT policy ($N = 1{,}000$ paths), now including the DeepAR probabilistic forecaster \citep{Salinas2020deepar}. Moving Average ($p = 10$) and Exponential Smoothing ($\alpha = 0.3$) both improve fill rate at E4 relative to the na\"ive forecast (from 80\% to 82\% and 83\%) but slightly increase cumulative BWR (from 427 to 446 and 443) and total cost (from 1,743 to 1,960 and 1,908). DeepAR, trained on 200 synthetic demand paths with a global model across all series, achieves the lowest total cost (1,718) among all forecasters while maintaining comparable BWR (428.0) and fill rate (80\%). The neural forecaster's advantage comes from learning cross-series demand patterns that reduce order variability at upstream echelons, though at the expense of a lower fill rate at E1 (50\% vs.\ 70--80\% for classical methods). This pattern confirms the accuracy--cost disconnect of \citet{Ban2019} in a multi-echelon setting: the choice of forecaster cannot be evaluated on point accuracy or on BWR alone but must be assessed through its downstream cost consequence across the full chain.

\begin{table}
\centering
\caption{Experiment 5b: Forecaster comparison at E4, OUT policy, $N = 1{,}000$ paths.}\label{tab:bench_fc}
\begin{tabular}{lrrrr}
\toprule
Forecaster & Cum.\ BWR & Fill Rate & Total Cost \\
\midrule
Na\"ive (constant) & 426.9 & 80\% & 1,743 \\
Moving Avg.\ ($p{=}10$) & 446.4 & 82\% & 1,960 \\
Exp.\ Smoothing ($\alpha{=}0.3$) & 442.6 & 83\% & 1,908 \\
DeepAR & 428.0 & 80\% & 1,718 \\
\bottomrule
\end{tabular}
\end{table}

\begin{figure}
\centering
\includegraphics[width=\textwidth]{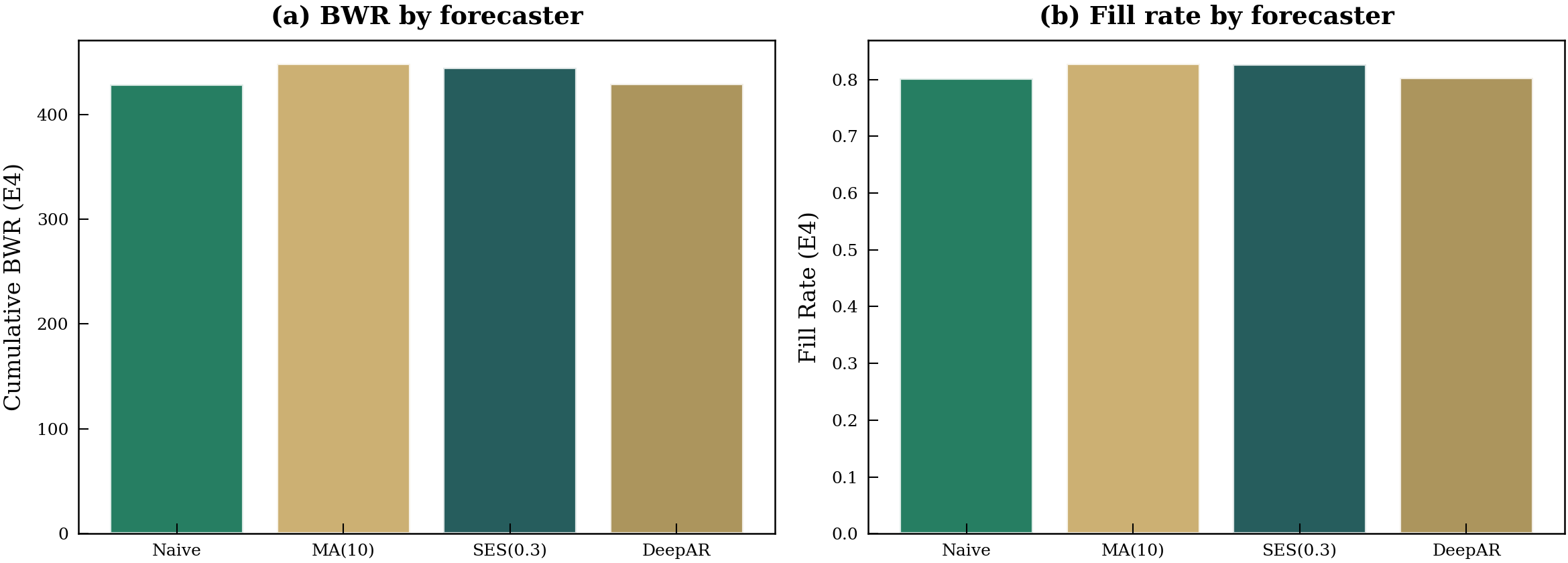}
\caption{Forecaster comparison at E4 under OUT policy ($N = 1{,}000$): (a)~cumulative BWR and (b)~fill rate. DeepAR achieves the lowest total cost while maintaining comparable BWR and fill rate to classical methods.}\label{fig:forecaster}
\end{figure}

\subsubsection{Experiment 5c: Cross-chain comparison} 

Applying the OUT policy across three predefined configurations ($N = 500$) reveals that chain structure drives bullwhip severity independently of the ordering policy. The semiconductor chain produces cumulative BWR of 429, the Beer Game chain 10.0, and the consumer two-tier chain 2.5 (Table~\ref{tab:bench_chain}). The 43-fold difference between the semiconductor and Beer Game chains arises entirely from lead time heterogeneity ($L = [2,4,12,8]$ vs.\ uniform $L = 2$), confirming the quadratic lead time dependence of Eq.~\eqref{eq:out_level} compounded through the multiplicative cascade of Eq.~\eqref{eq:cum_bwr}.

\begin{table}
\centering
\caption{Experiment 5c: Cross-chain comparison, OUT policy, $N = 500$ paths.}\label{tab:bench_chain}
\begin{tabular}{lcrr}
\toprule
Chain configuration & Echelons & Cum.\ BWR & Total Cost \\
\midrule
Semiconductor 4-tier & 4 & 429 & 3,560 \\
Beer Game 4-tier & 4 & 10.0 & 1,193 \\
Consumer 2-tier & 2 & 2.5 & 483 \\
\bottomrule
\end{tabular}
\end{table}

\subsubsection{Experiment 5d: Real versus synthetic demand} 

The most consequential benchmark result comes from replaying 60 months of WSTS semiconductor billings through the chain using the \texttt{ReplayDemandGenerator} ($N = 500$ paths). Each path replays the historical series with independent 5\% Gaussian noise added to generate path diversity for Monte Carlo evaluation; a single noiseless path produces qualitatively similar results. Under the OUT policy, real demand produces cumulative BWR of 66,076, compared with 429 under synthetic AR(1), a gap of approximately $155\times$ (Table~\ref{tab:bench_real}, Fig.~\ref{fig:real}). This gap arises because the WSTS data exhibits higher autocorrelation, regime-switching behavior (the 2019 trade-war downturn, the 2020--2021 pandemic surge), and structural breaks that the stationary AR(1) model does not capture. The POUT policy with $\alpha = 0.5$ reduces the real-data BWR from 66,076 to 857, confirming its effectiveness under realistic conditions. But the magnitude of the gap itself is the central finding: conclusions drawn from synthetic benchmarks may underestimate operational bullwhip risk in semiconductor supply chains by two orders of magnitude, motivating the inclusion of real-data replay capabilities in the benchmarking framework.

\begin{table}
\centering
\caption{Experiment 5d: Cumulative BWR and total cost at E4 under synthetic AR(1) vs.\ real WSTS demand ($N = 500$).}\label{tab:bench_real}
\begin{tabular}{lrrrr}
\toprule
& \multicolumn{2}{c}{Cum.\ BWR} & \multicolumn{2}{c}{Total Cost} \\
\cmidrule(lr){2-3}\cmidrule(lr){4-5}
Policy & Synthetic & WSTS & Synthetic & WSTS \\
\midrule
OUT & 429 & 66,076 & 3,560 & 34,892 \\
POUT ($\alpha{=}0.5$) & 16 & 857 & 8,113 & 8,088 \\
\bottomrule
\end{tabular}
\end{table}

\begin{figure}
\centering
\includegraphics[width=0.7\textwidth]{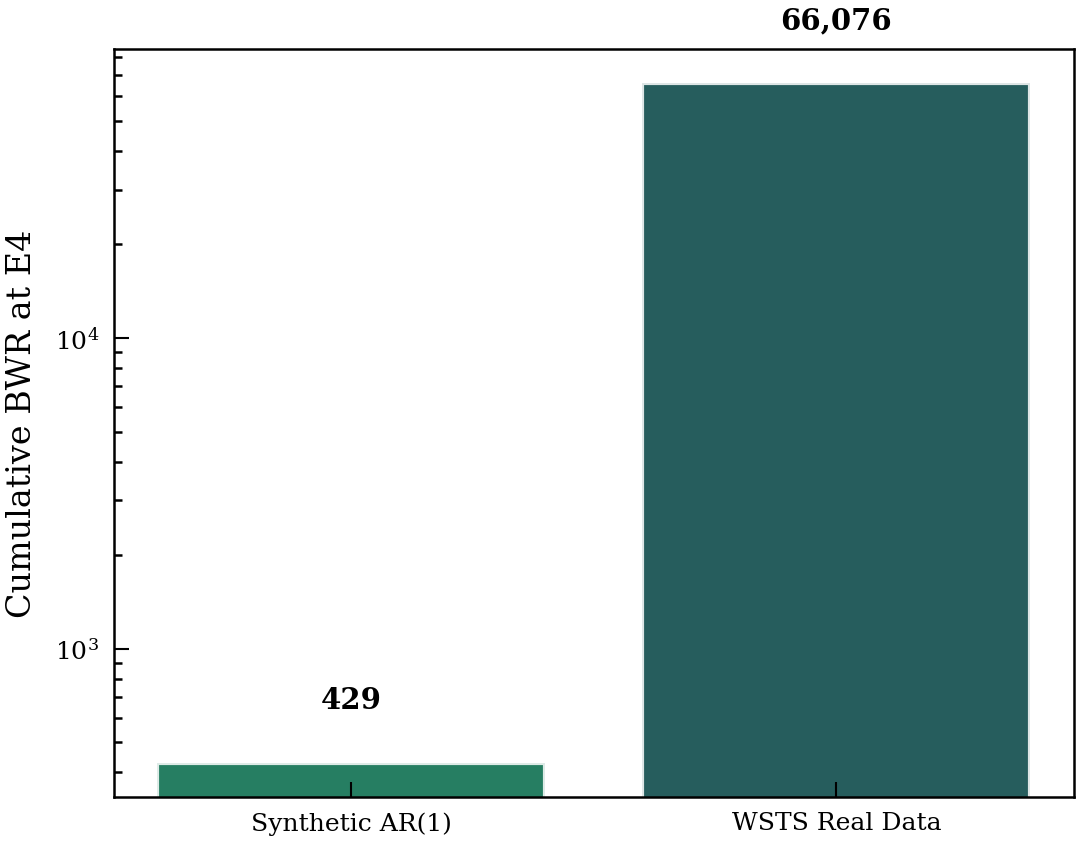}
\caption{Cumulative BWR at E4 under OUT: synthetic AR(1) vs.\ real WSTS demand. The $155\times$ gap illustrates the importance of empirical benchmarking.}\label{fig:real}
\end{figure}

\section{Discussion}\label{sec:discussion}

We discuss the experimental findings and their implications and identify the limitations of the present work together with directions for future research.

\subsection{Managerial implications}\label{sec:disc_managerial}

The experimental findings carry four implications, especially for semiconductor supply chain managers.

First, lead time compression at the bottleneck tier yields disproportionate returns. Reducing foundry lead time from 12 to 8 weeks (a 33\% reduction in a single parameter) produces an order-of-magnitude improvement in cumulative BWR, whereas eliminating bullwhip entirely at the distributor tier reduces it by only 11\%. This asymmetry supports the semiconductor industry's ongoing investments in advanced packaging and regionalized fabrication capacity \citep{SIA2025}, both of which shorten the effective pipeline without requiring changes to downstream ordering practices. Second, the stochastic filtering result (Proposition~\ref{prop:filtering}) and its cumulative extension (Corollary~\ref{cor:cumulative}) imply that upstream capacity planning at foundries and wafer suppliers can rely on deterministic scenario analysis parameterized by chain configuration, without requiring computationally expensive Monte Carlo simulation. The cumulative amplification itself concentrates ($\mathrm{CV}(\mathrm{BWR}_{\mathrm{cum}}) = 0.236$), and total supply chain cost concentrates even more strongly ($\mathrm{CV}(\mathrm{TC}) = 0.079$), because the majority of cost is incurred at the near-deterministic upstream tiers. The distributor tier, by contrast, faces path-dependent risk (CV of BWR $= 0.21$) that cannot be characterized without stochastic evaluation. This asymmetry suggests that different planning methodologies are appropriate at different supply chain tiers.

Next, the BWR--NSAmp tradeoff revealed by the policy comparison (Table~\ref{tab:bench_policy}) and the POUT Pareto frontier (Fig.~\ref{fig:pareto}) implies that practitioners should not adopt smoothing policies without evaluating their inventory variance consequences across the full chain. POUT ($\alpha = 0.3$) reduces cumulative BWR by 96\% but collapses the fill rate to near zero at E1; Smoothing OUT preserves fill rates but quadruples E4 inventory variance. The cost-asymmetry experiment (Table~\ref{tab:cost_policy}) further demonstrates that policy rankings are not robust to the cost structure: POUT dominates at low backorder-to-holding ratios ($b/h = 2$), but OUT becomes optimal at $b/h \geq 5$, and Smoothing OUT overtakes POUT at $b/h = 20$. The choice of $\alpha$ and the choice of policy class are both managerial decisions that the benchmarking framework supports directly through systematic sensitivity analysis. Finally, the 155-fold gap between synthetic AR(1) and real WSTS bullwhip severity carries a cautionary message: inventory planning tools validated on synthetic demand may produce substantially optimistic risk assessments when applied to real semiconductor demand with its regime-switching dynamics. The availability of real-data replay in the benchmarking framework allows practitioners to validate ordering policies against historical demand before deployment.

\subsection{Limitations and future research}\label{sec:disc_limitations}

The experiments in this paper use serial chain topologies, though the package has been extended to support convergent and divergent structures via \texttt{SupplyChainGraph} and \texttt{NetworkSupplyChain}, supporting arbitrary directed acyclic graph topologies with NetworkX integration. Evaluating the benchmarking framework on non-serial topologies is a natural next step. The current implementation assumes unlimited upstream capacity, precluding study of the rationing game \citep{Lee1997}. Demand models in the catalog are univariate; multi-product demand with substitution effects would require extending the demand generator ABC to produce multivariate time series. The vectorized engine's $O(NKT)$ memory scaling limits single-machine experiments to approximately $N = 50{,}000$ and $K = 20$; larger configurations would require chunked processing. The DeepAR probabilistic forecaster \citep{Salinas2020deepar} has been integrated as a first deep learning baseline, demonstrating that the registry architecture supports neural forecasters without code modification. Expanding the catalog further to include deep reinforcement learning policies \citep{Boute2022} and end-to-end newsvendor methods \citep{Oroojlooyjadid2020} would increase the framework's value as a community resource. Because the registry architecture supports extension without code modification, these additions can be contributed incrementally.

\section{Conclusion}\label{sec:conclusion}

The bullwhip effect remains operationally persistent in part because the computational infrastructure for studying it has not kept pace with advances in analytical theory. This paper introduced \texttt{deepbullwhip}, an open-source Python package comprising a modular simulation engine with pluggable demand generators, ordering policies, cost functions, and forecasters, together with a vectorized Monte Carlo engine (20.8 million cells in under 7 seconds) and a registry-based benchmarking framework shipping a curated catalog of components and demand datasets. Five sets of experiments on a four-echelon semiconductor supply chain produced four principal insights: cumulative amplification of $427\times$ arising from a multiplicative cascade across echelons; a stochastic filtering phenomenon whereby the cross-path CV of BWR drops from 0.21 at the distributor to 0.01 at the foundry, formally explained by Proposition~\ref{prop:filtering} and extended to the cumulative BWR by Corollary~\ref{cor:cumulative} (which predicts $\mathrm{CV}(\mathrm{BWR}_\mathrm{cum})$ within 1.8\%); a quantitative BWR--NSAmp tradeoff whose Pareto frontier reveals non-monotonic inventory variance responses to order smoothing and whose policy rankings reverse under different cost asymmetries; and a 155-fold gap between synthetic and real-data bullwhip severity that underscores the importance of empirical benchmarking. The simulator's correctness is confirmed by matching the Chen et al.\ (2000) analytical bound to within 0.14\% under the bound's own assumptions. Future work will evaluate the recently added non-serial topologies (available from v0.3.0) against the benchmarking framework, add capacity constraints and multi-product demand, and expand the catalog through community contributions. The package is available at \url{https://github.com/ai-vnv/deepbullwhip} with full documentation and leaderboard to encourage community contributions.



\section*{Acknowledgment}

This work was supported by the Interdisciplinary Research Center for Smart Mobility \& Logistics (IRC-SML) at King Fahd University of Petroleum \& Minerals (KFUPM).







 \bibliographystyle{elsarticle-num-names} 
 \bibliography{reference}







\appendix
\newpage

\section{Proof of Proposition~\ref{prop:filtering}}\label{app:proof_filtering}

\begin{proof}
Let $X = \mathrm{Var}(O_k \mid \omega)$ and $Y = \mathrm{Var}(O_{k-1} \mid \omega)$ be positive random variables (over the probability space of demand realizations $\omega$) with means $\mu_X = \mathbb{E}[X]$, $\mu_Y = \mathbb{E}[Y]$, standard deviations $\sigma_X$, $\sigma_Y$, and correlation $\rho_{XY}$. Define $g(X,Y) = X/Y$. By the multivariate delta method, for a twice-differentiable function $g$ evaluated at the mean $(\mu_X, \mu_Y)$:
\begin{equation}\label{eq:delta_var}
\mathrm{Var}\!\bigl(g(X,Y)\bigr) \;\approx\; \nabla g^\top \Sigma \, \nabla g,
\end{equation}
where $\nabla g = \bigl(\partial g/\partial X,\; \partial g/\partial Y\bigr)^\top$ evaluated at $(\mu_X, \mu_Y)$, and $\Sigma$ is the covariance matrix of $(X, Y)$.

Computing the partial derivatives of $g(X,Y) = X/Y$:
\[
\frac{\partial g}{\partial X}\bigg|_{(\mu_X,\mu_Y)} = \frac{1}{\mu_Y}, \qquad
\frac{\partial g}{\partial Y}\bigg|_{(\mu_X,\mu_Y)} = -\frac{\mu_X}{\mu_Y^2}.
\]
Substituting into Eq.~\eqref{eq:delta_var} with $\Sigma_{11} = \sigma_X^2$, $\Sigma_{22} = \sigma_Y^2$, $\Sigma_{12} = \rho_{XY}\sigma_X\sigma_Y$:
\begin{align}
\mathrm{Var}\!\left(\frac{X}{Y}\right) &\approx \frac{\sigma_X^2}{\mu_Y^2} + \frac{\mu_X^2 \sigma_Y^2}{\mu_Y^4} - \frac{2\mu_X \rho_{XY}\sigma_X\sigma_Y}{\mu_Y^3}. \label{eq:var_expand}
\end{align}
Dividing both sides by $\bigl(\mathbb{E}[g]\bigr)^2 \approx \mu_X^2/\mu_Y^2$ and writing $\mathrm{CV}(X) = \sigma_X/\mu_X$ and $\mathrm{CV}(Y) = \sigma_Y/\mu_Y$:
\begin{equation}
\mathrm{CV}\!\left(\frac{X}{Y}\right)^2 \;\approx\; \mathrm{CV}(X)^2 + \mathrm{CV}(Y)^2 - 2\,\rho_{XY}\,\mathrm{CV}(X)\,\mathrm{CV}(Y). \label{eq:cv_squared}
\end{equation}
Taking the square root yields Eq.~\eqref{eq:cv_ratio}. When $\mathrm{CV}(X) = \mathrm{CV}(Y) = c$ (which holds approximately in the semiconductor chain, where both order streams are amplified by the same cascade), the expression simplifies to
\[
\mathrm{CV}\!\left(\frac{X}{Y}\right) \;\approx\; c\,\sqrt{2(1 - \rho_{XY})},
\]
which tends to zero as $\rho_{XY} \to 1$, regardless of $c$. The approximation is first-order accurate in the CVs, requiring $\mathrm{CV}(X), \mathrm{CV}(Y) \ll 1$; in the semiconductor chain configuration, the empirical CVs of $\mathrm{Var}(O_k)$ are approximately 0.15, and the prediction matches empirical values to four decimal places (Table~\ref{tab:filtering_validation}).
\end{proof}

\section{Proof of Corollary~\ref{cor:cumulative}}\label{app:proof_cumulative}

\begin{proof}
Let $Z_k = \mathrm{BWR}_k$ for $k = 1,\ldots,K$, and define $g(Z_1,\ldots,Z_K) = \prod_{k=1}^K Z_k$. Evaluating at the means $\boldsymbol{\mu} = (\mu_1,\ldots,\mu_K)$:
\[
\frac{\partial g}{\partial Z_j}\bigg|_{\boldsymbol{\mu}} = \prod_{k \neq j} \mu_k = \frac{\mu_{\mathrm{cum}}}{\mu_j}, \quad \text{where } \mu_{\mathrm{cum}} = \prod_{k=1}^K \mu_k.
\]
By the multivariate delta method, $\mathrm{Var}(g) \approx \nabla g^\top \Sigma \nabla g$, which yields
\[
\mathrm{Var}\!\bigl(\mathrm{BWR}_{\mathrm{cum}}\bigr) \;\approx\; \sum_{j=1}^K \frac{\mu_{\mathrm{cum}}^2}{\mu_j^2}\,\sigma_j^2 + 2\!\!\sum_{1 \le i < j \le K}\!\! \frac{\mu_{\mathrm{cum}}^2}{\mu_i\,\mu_j}\,\mathrm{Cov}(Z_i, Z_j).
\]
Dividing by $\mathbb{E}[\mathrm{BWR}_{\mathrm{cum}}]^2 \approx \mu_{\mathrm{cum}}^2$ and writing $\mathrm{CV}(Z_k) = \sigma_k/\mu_k$ gives Eq.~\eqref{eq:cv_cum}. The dominance by the first $k^*{-}1$ terms follows because $\mathrm{CV}(\mathrm{BWR}_k) \approx 0$ for $k \ge k^*$ drives all terms involving those indices to zero.
\end{proof}

\section{Algorithm pseudocode}\label{app:algorithms}

\begin{algorithm}[htbp]
\caption{Model registry: registration and retrieval}\label{alg:registry}
\begin{algorithmic}[1]
\State $\mathcal{R} \gets$ empty dictionary of dictionaries \Comment{Global registry}
\Procedure{Register}{category, name, class}
  \State $\mathcal{R}[\text{category}][\text{name}] \gets \text{class}$
  \State Attach metadata $(\text{category}, \text{name})$ to class
\EndProcedure
\Procedure{Get}{category, name, params}
  \State $\text{cls} \gets \mathcal{R}[\text{category}][\text{name}]$
  \State \Return $\text{cls}(\text{params})$ \Comment{Instantiate with keyword arguments}
\EndProcedure
\end{algorithmic}
\end{algorithm}

\newpage

\section{Code examples}\label{app:code}

Listing~\ref{lst:register} illustrates the custom decorator pattern described in Section~\ref{sec:registry}. 

\begin{lstlisting}[caption={Registering a custom Proportional OUT policy.},label={lst:register}]
from deepbullwhip.registry import register
from deepbullwhip.policy.base import OrderingPolicy

@register("policy", "my_pout")
class MyPOUT(OrderingPolicy):
    def __init__(self, lead_time, service_level=0.95,
                 alpha=0.5):
        self.alpha = alpha
        # ... (store lead_time, compute z_alpha)

    def compute_order(self, ip, fm, fs):
        S = (self.lead_time + 1) * fm + (
            self.z_alpha * fs * (self.lead_time+1)**0.5)
        return max(0.0, (S - ip) * self.alpha)
\end{lstlisting}






Listing~\ref{lst:extend} demonstrates extensibility through three use cases: replaying historical demand, implementing a custom cost function with a penalty, and running a minimal simulation.

\begin{lstlisting}[caption={Extensibility examples: custom demand replay, cost, and minimal simulation.},label={lst:extend}]
# --- Replay historical WSTS data ---
from deepbullwhip.datasets import load_wsts
from deepbullwhip.demand.replay import ReplayDemandGenerator
wsts = load_wsts()
gen = ReplayDemandGenerator(data=wsts)

# --- Use perishable cost function ---
from deepbullwhip.cost.perishable import PerishableCost
cost = PerishableCost(holding_cost=0.15,
    backorder_cost=0.60, gamma=0.05, buffer=50)

# --- Minimal simulation (5 lines) ---
from deepbullwhip import *
import numpy as np
gen = SemiconductorDemandGenerator()
d = gen.generate(T=156, seed=42)
chain = SerialSupplyChain()
result = chain.simulate(d, np.full_like(d, d.mean()),
                           np.full_like(d, d.std()))
\end{lstlisting}

\newpage

\section{Scalability benchmarks}\label{app:scalability}

The scalability benchmarks results are shown in Tables~\ref{tab:scaling_N}, \ref{tab:scaling_T}, \ref{tab:scaling_K}, and \ref{tab:stress}.

\begin{table}
\centering
\caption{Scaling with $N$ ($T = 156$, $K = 4$).}\label{tab:scaling_N}
\begin{tabular}{rrrr}
\toprule
$N$ & Serial (s) & Vec.\ (s) & Speedup \\
\midrule
10 & 0.10 & 0.011 & $9\times$ \\
100 & 0.97 & 0.019 & $50\times$ \\
500 & 4.62 & 0.067 & $69\times$ \\
1,000 & 9.40 & 0.142 & $66\times$ \\
5,000 & 44.68 & 0.741 & $60\times$ \\
\bottomrule
\end{tabular}
\end{table}

\begin{table}
\centering
\caption{Scaling with $T$ ($N = 500$, $K = 4$).}\label{tab:scaling_T}
\begin{tabular}{rlrrr}
\toprule
$T$ & Horizon & Serial (s) & Vec.\ (s) & Speedup \\
\midrule
52 & 1 year & 1.50 & 0.032 & $47\times$ \\
156 & 3 years & 4.64 & 0.057 & $82\times$ \\
520 & 10 years & 15.89 & 0.254 & $63\times$ \\
1,040 & 20 years & 31.81 & 0.589 & $54\times$ \\
5,200 & 100 years & 148.69 & 3.076 & $48\times$ \\
\bottomrule
\end{tabular}
\end{table}

\begin{table}
\centering
\caption{Scaling with $K$ ($N = 500$, $T = 156$).}\label{tab:scaling_K}
\begin{tabular}{rrrrl}
\toprule
$K$ & Serial (s) & Vec.\ (s) & Speedup & $\text{BWR}_{\text{cum}}$ \\
\midrule
2  & 2.40  & 0.024 & $99\times$ & $5.5 \times 10^{1}$ \\
4  & 4.42  & 0.053 & $83\times$ & $1.5 \times 10^{3}$ \\
8  & 9.98  & 0.138 & $72\times$ & $2.4 \times 10^{8}$ \\
12 & 18.02 & 0.208 & $87\times$ & $8.3 \times 10^{14}$ \\
16 & 22.23 & 0.306 & $73\times$ & $9.4 \times 10^{21}$ \\
\bottomrule
\end{tabular}
\end{table}

\begin{table}
\centering
\caption{Stress test results (vectorized engine).}\label{tab:stress}
\begin{tabular}{rrrrrr}
\toprule
$N$ & $T$ & $K$ & Cells ($\times 10^6$) & Time (s) & Mem (MB) \\
\midrule
1,000 & 520 & 4 & 2.08 & 0.45 & 48 \\
5,000 & 156 & 4 & 3.12 & 0.62 & 71 \\
1,000 & 156 & 12 & 1.87 & 0.43 & 43 \\
10,000 & 156 & 4 & 6.24 & 1.55 & 143 \\
5,000 & 520 & 8 & 20.80 & 6.81 & 476 \\
\bottomrule
\end{tabular}
\end{table}

\newpage

\section{Diagnostic plots}\label{app:diagnostics}

Figures~\ref{fig:app_orders} through \ref{fig:app_wafer} provide detailed diagnostic views of the Experiment~1 simulation ($T = 156$, constant forecast). 

\begin{figure}
\centering
\includegraphics[width=\textwidth]{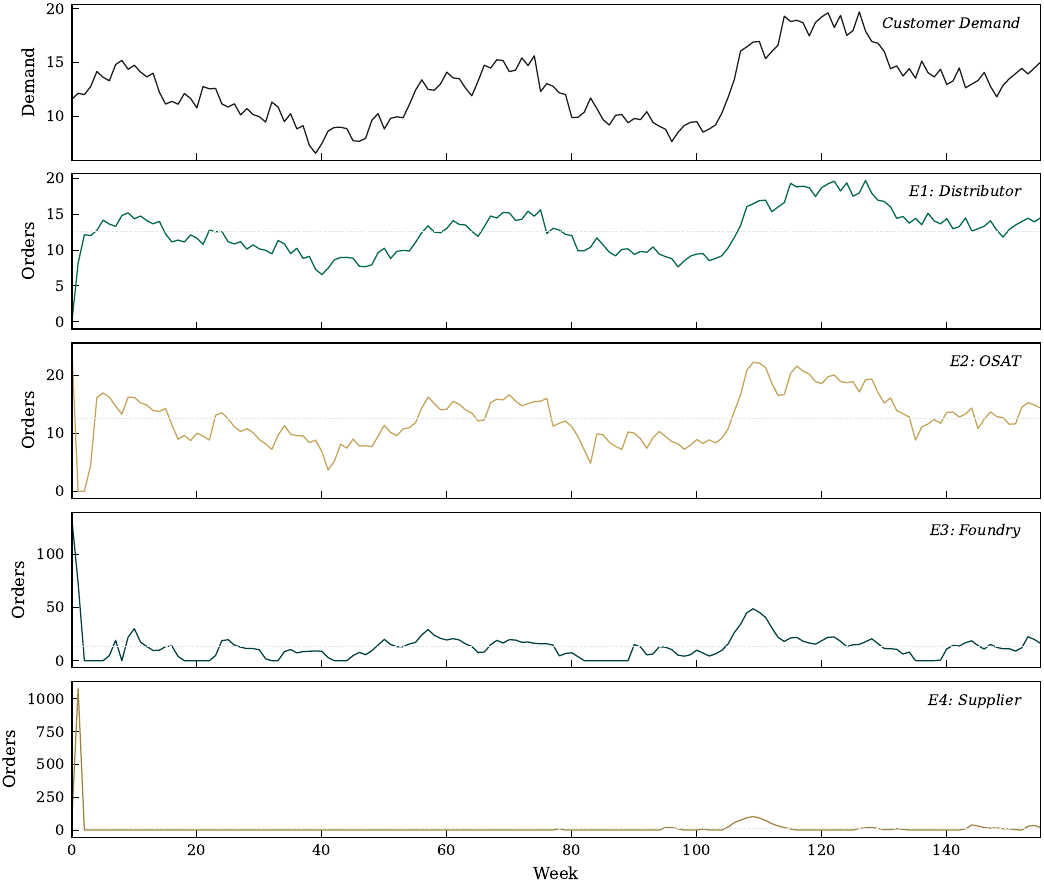}
\caption{Order quantities across all echelons. Consumer demand (top) is progressively amplified at each tier, with E4 order spikes exceeding 1,000 units.}\label{fig:app_orders}
\end{figure}

\begin{figure}
\centering
\includegraphics[width=\textwidth]{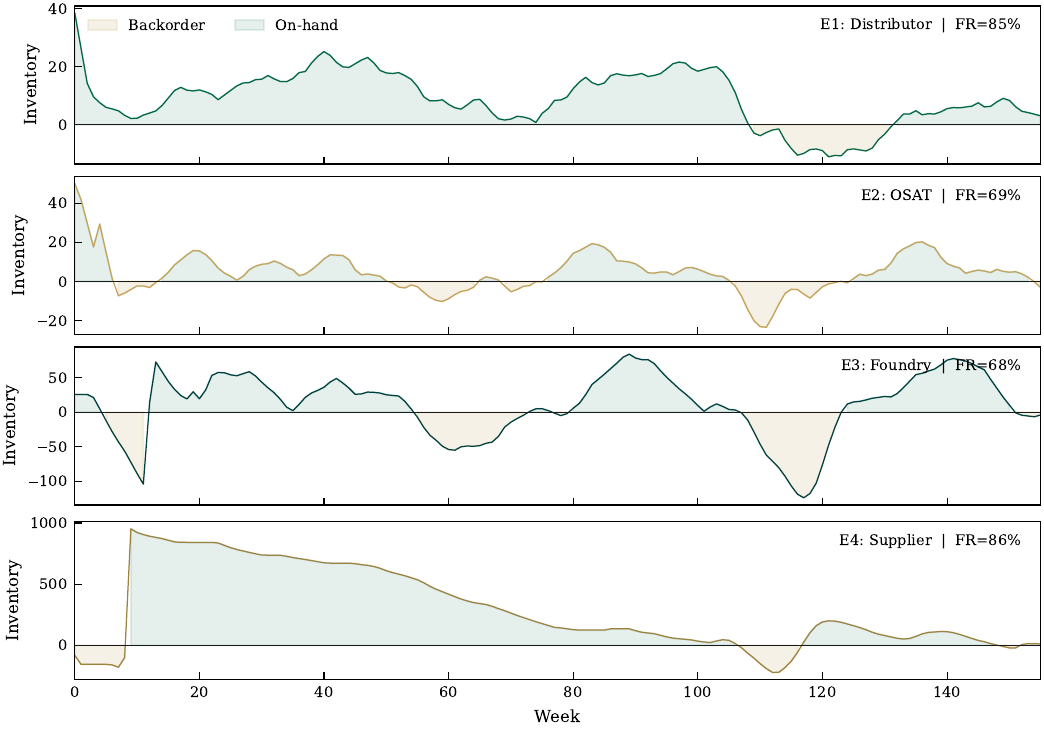}
\caption{Net inventory levels by echelon. Positive values represent on-hand stock; negative values represent backorders. E4 exhibits the largest excursions.}\label{fig:app_inv}
\end{figure}

\begin{figure}
\centering
\includegraphics[width=\textwidth]{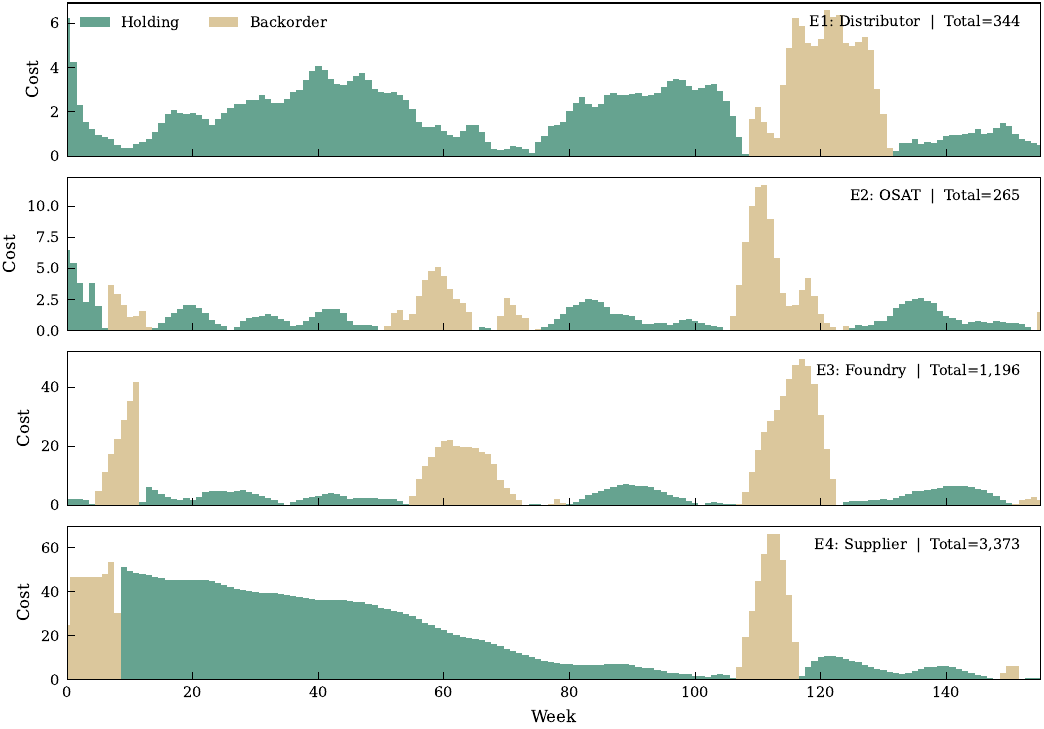}
\caption{Per-period cost decomposition by echelon showing the asymmetric holding and backorder components.}\label{fig:app_cost}
\end{figure}

\begin{figure}
\centering
\includegraphics[width=0.9\textwidth]{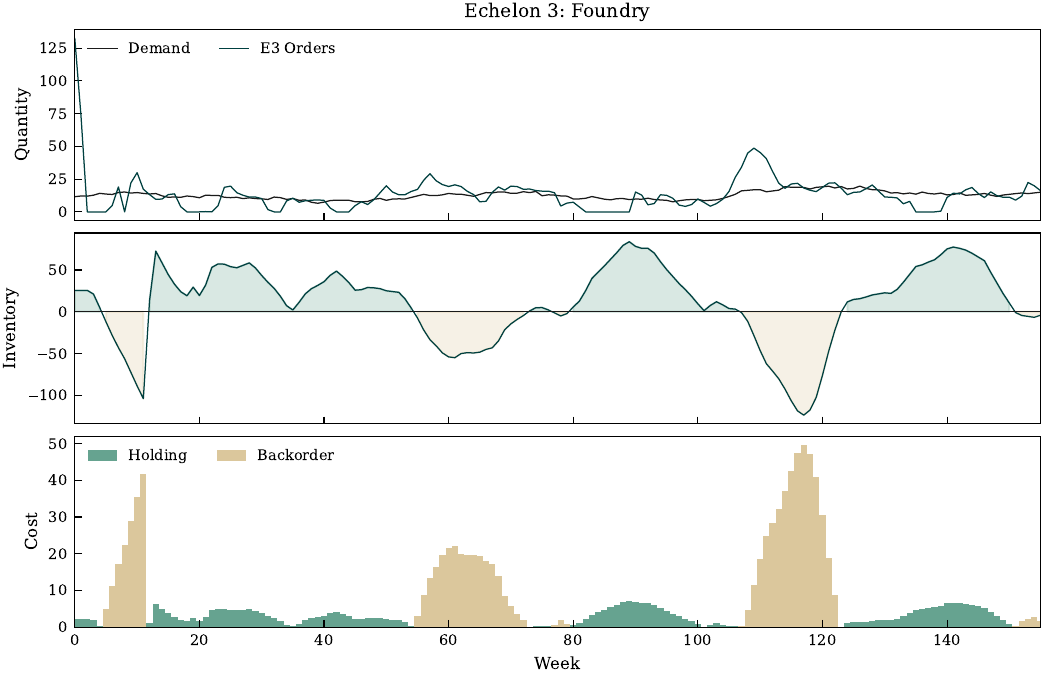}
\caption{Echelon detail for E3 (Foundry, $L_3 = 12$ weeks): demand, orders, inventory, and cost.}\label{fig:app_foundry}
\end{figure}

\begin{figure}
\centering
\includegraphics[width=0.9\textwidth]{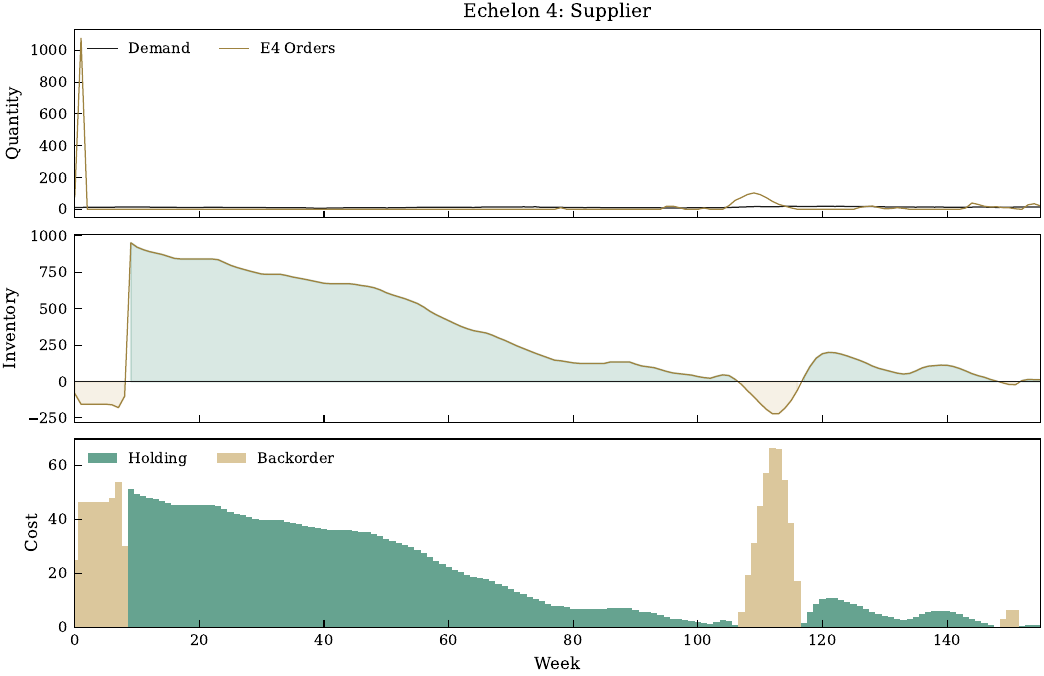}
\caption{Echelon detail for E4 (Wafer Supplier, $L_4 = 8$ weeks).}\label{fig:app_wafer}
\end{figure}

\section{Cost sensitivity}\label{app:sensitivity}

Finally, the total supply chain cost and cost ratio $b/h$ is shown in Figure~\ref{fig:cost_ratio}.

\begin{figure}
\centering
\includegraphics[width=0.7\textwidth]{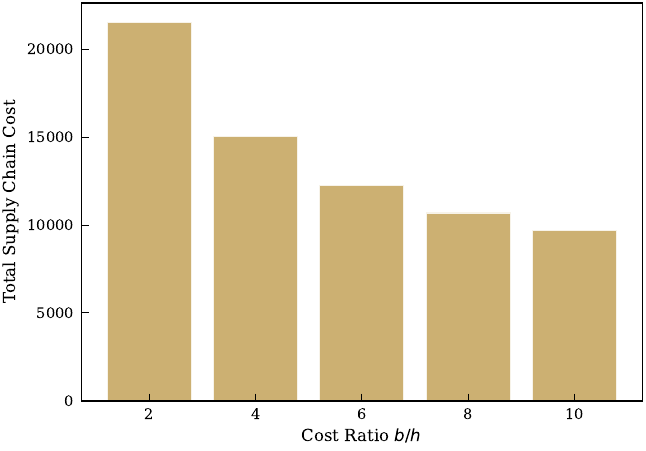}
\caption{Total supply chain cost vs.\ cost ratio $b/h$ under normalized costs ($h + b = 1$). Higher ratios concentrate cost in rare but extreme stockout events.}\label{fig:cost_ratio}
\end{figure}

\end{document}